\documentclass[12pt]{amsart}

\usepackage{etex}
\usepackage{amsmath, amssymb}
\usepackage{array}
\usepackage[frame,cmtip,arrow,matrix,line,graph,curve]{xy}
\usepackage{graphpap, color, paralist, pstricks}
\usepackage[mathscr]{eucal}
\usepackage[pdftex]{graphicx}
\usepackage[pdftex,colorlinks,backref=page,citecolor=blue]{hyperref}
\usepackage{tikz}
\usepackage{tikz-cd}

\newtheorem{theorem}{Theorem}[section]
\newtheorem{theoremletter}{Theorem}

\newtheorem{proposition}[theorem]{Proposition}
\newtheorem{corollary}[theorem]{Corollary}
\newtheorem{lemma}[theorem]{Lemma}

\theoremstyle{definition}
\newtheorem{definition}[theorem]{Definition}

\newtheorem{question}[theorem]{Question}
\newtheorem{remark}[theorem]{Remark}

\newtheorem*{acknowledgement}{Acknowledgement}

\renewcommand{\AA}{\mathbb{A} }
\newcommand{\CC}{\mathbb{C} }
\newcommand{\LL}{\mathbb{L} }
\newcommand{\NN}{\mathbb{N} }
\newcommand{\PP}{\mathbb{P} }

\newcommand{\RR}{\mathbb{R} }

\newcommand{\ZZ}{\mathbb{Z} }

\newcommand{\cE}{\mathcal{E} }
\newcommand{\cF}{\mathcal{F} }

\newcommand{\cM}{\mathcal{M} }

\newcommand{\cO}{\mathcal{O} }

\newcommand{\cS}{\mathcal{S} }

\newcommand{\rD}{\mathrm{D} }
\newcommand{\rH}{\mathrm{H} }
\newcommand{\rM}{\mathrm{M} }

\newcommand{\ba}{\mathbf{a} }
\newcommand{\bfb}{\mathbf{b} }
\newcommand{\bc}{\mathbf{c} }

\newcommand{\bp}{\mathbf{p} }
\newcommand{\bq}{\mathbf{q} }

\newcommand{\bQ}{\mathbf{Q} }

\newcommand{\bR}{\mathbf{R} }

\def\rHom{\mathrm{Hom} }

\def\SL{\mathrm{SL}}

\def\Gr{\mathrm{Gr}}
\def\Fl{\mathrm{Fl}}

\def\git{/\!/ }
\def\Pic{\mathrm{Pic} }

\def\im{\mathrm{im}\, }

\def\Hilb{\mathrm{Hilb}}

\def\pdeg{\mathrm{pardeg}}

\def\Jac{\mathrm{Jac}}
\def\Fdim{\mathrm{Fdim}\;}

\begin{document}

\title[Derived categories of symmetric products and moduli spaces of bundles]{Derived categories of symmetric products and \\ moduli spaces of vector bundles on a curve}
\date{\today}

%\keywords{}
%\subjclass[2010]{05C05, 14M15, 14N10, 14T15}

\author{Kyoung-Seog Lee}
\address{Kyoung-Seog Lee, Department of Mathematics, POSTECH, 77, Cheongam-ro, Nam-gu, Pohang-si, Gyeongsangbuk-do, 37673, Korea} \email{kyoungseog@postech.ac.kr}

\author{Han-Bom Moon}
\address{Han-Bom Moon, Department of Mathematics, Fordham University, New York, NY 10023}
\email{hmoon8@fordham.edu}

\maketitle

\begin{abstract}
We show that the derived categories of symmetric products of a curve are embedded into the derived categories of the moduli spaces of vector bundles of large ranks on the curve. It supports a prediction of the existence of a semiorthogonal decomposition of the derived category of the moduli space, expected by a motivic computation. As an application, we show that all Jacobian varieties, symmetric products of curves and all principally polarized abelian varieties of dimension at most three, are Fano visitors. We also obtain similar results for motives.
\end{abstract}

\section{Introduction}\label{sec:intro}

Let $X$ be a smooth projective curve of genus $g \ge 2$, and $L$ be a line bundle on $X$ of degree $d$. The moduli space $\rM_{X}(r, L)$ of rank $r$, determinant $L$ semistable vector bundles on $X$ is one of the most intensively studied moduli spaces in the past decades. When $(r, d) = 1$, it is a smooth projective Fano variety of dimension $(r^{2}-1)(g-1)$ of index two \cite{Ram73}. 

\subsection{Derived category of the moduli space of vector bundles}

Since the pioneering works of Narasimhan in \cite{Nar17, Nar18} and Fonarev-Kuznetsov in \cite{FK18}, many works have been done to understand the bounded derived category of coherent sheaves $\rD^{b}(\rM_{X}(r, L))$ of the moduli space, in particular its semiorthogonal decomposition. Narasimhan, and independently Belmans--Galkin--Mukhopadhyay, conjectured that in the rank two case, $\rD^{b}(\rM_{X}(2, L))$ has an explicit semiorthogonal decomposition \cite{Lee18, BGM23}, where all indecomposable components are equivalent to $\rD^{b}(X_{n})$, where $X_{n} = X^{n}/S_{n}$. A proof of this conjecture was recently announced by Tevelev and Torres \cite{TT21, Tev23}. 

For the higher rank case, based on motivic computation in \cite{GL20}, it has been conjectured that $\rD^{b}(\rM_{X}(r, L))$ has a semiorthogonal decomposition whose components can be described in terms of symmetric products $X_{n}$ and its Jacobian $\mathrm{Jac}(X)$ \cite[Conjecture 1.3]{GL20}. As the initial step, building upon earlier works of Narasimhan in \cite{Nar17, Nar18}, Fonarev-Kuznetsov in \cite{FK18}, Belmans-Mukhopadhyay in \cite{BM19}, we proved that $\rD^{b}(X)$ can be embedded into $\rD^{b}(\rM_{X}(r, L))$ in \cite{LM21, LM23} for any curve $X$, rank $r \geq 2$ and coprime degree $d$. In this paper, we extend this result to symmetric products. 

\begin{theoremletter}\label{thm:embedding}
Suppose that $r > 2n$. Then $\rD^{b}(X_{n})$ is embedded into $\rD^{b}(\rM_{X}(r, L))$. 
\end{theoremletter}

As we can see below, this result implies that $\rD^{b}(\mathrm{Jac}(X))$ is embedded into $\rD^{b}(\rM_{X}(r, L))$ if $r > 2g$. 

\subsection{Fano visitor problem}

Mirror symmetry predicts that a mirror of a Fano variety is given by a Landau-Ginzburg model, and people have tried to understand Fano varieties via their Landau-Ginzburg mirrors. From this perspective, it is essential to know which categories can appear as semiorthogonal components of the derived categories of Fano varieties since we expect they will also appear as Fukaya-Seidel categories associated with some critical loci of the Landau-Ginzburg mirrors. 

On the other hand, studying semiorthogonal decompositions of derived categories of Fano varieties has played an important role in the theory of derived categories. It has many interesting (sometimes conjectural) consequences to birational geometry, especially rationality, mirror symmetry, moduli spaces of ACM/Ulrich bundles, motives, quantum cohomology, and other geometric properties of Fano varieties. When the derived category of a Fano variety contains the derived category of a projective variety, those two varieties are expected to interchange geometric information. See \cite{KKLL17, KL23} and references therein for more details. Therefore, it is an interesting question which categories can be embedded into the derived categories of Fano varieties.

\begin{definition}\label{def:Fanovisitor}
Let $V$ be a smooth projective variety and $W$ be a smooth projective Fano variety. If there is a fully faithful functor $\rD^{b}(V) \hookrightarrow \rD^{b}(W)$, we say $V$ is a \emph{Fano visitor}, and $W$ is a \emph{Fano host}. The smallest dimension of Fano hosts of $V$ is called the \emph{Fano dimension} of $V$ and denoted by $\Fdim V$. If there is no Fano host, we say $\Fdim V = \infty$. 
\end{definition}

In 2011, Bondal asked the following fundamental question:

\begin{question}[Bondal, \protect{\cite[Question 1.1]{BBF16}}]\label{que:Bondal}
Is every smooth projective variety a Fano visitor?
\end{question}

In other words, he asked if $\Fdim V < \infty$ for every smooth projective variety $V$. It was predicted by Orlov \cite[Conjecture 10]{Orl09} and recently proved by Olander \cite[Theorem 2]{Ola21} that $\rD^{b}(V) \hookrightarrow \rD^{b}(W)$ implies $\dim V \le \dim W$. This implies that $\dim V \le \Fdim V$ and clearly $\dim V = \Fdim V$ if $V$ is a Fano variety. Thus, one can use the \emph{Fano defect} $\Fdim V - \dim V$ to measure how the given variety $V$ is different from the class of Fano varieties. 

There are a few cases which the affirmative answer to the Fano visitor problem is known. Examples of Fano visitors include all curves \cite{Nar17, Nar18, FK18}, all complete intersections \cite{KKLL17}, Hirzebruch surfaces \cite{KL23}, and general Enriques surfaces \cite{Kuz19}. 

Theorem \ref{thm:embedding} immediately implies that for any genus $g \ge 2$ curve $X$ and $n \in \ZZ_{\ge 0}$, its symmetric product $X_{n}$ is a Fano visitor, and $\Fdim X_{n} \le ((2n+1)^{2}-1)(g-1)$. On the other hand, $\rD^{b}(\Jac(X))$ is embedded into $\rD^{b}(X_{g})$ (Section \ref{sec:Fanovisitor}). Thus, we obtain:

\begin{theoremletter}\label{thm:Fanovisitor}
For any nonsingular projective curve $X$ of genus $g \ge 2$, its Jacobian $\Jac(X)$ is a Fano visitor, and $\Fdim \Jac(X) \le \dim \rM_{X}(2g+1, L) = 4(g+1)g(g-1)$. 
\end{theoremletter}

\subsection{Idea of proof}

Let $\cM_{X}(r, L)$ be the moduli stack of rank $r$, determinant $L$ vector bundles on $X$. Let $\cE$ be the universal bundle over $X \times \cM_{X}(r, L)$. Choosing a section $\sigma : \rM_{X}(r, L) \to \cM_{X}(r, L)^{s}$, we have a \emph{Poincar\'e bundle} $\sigma^{*}\cE$ over $X \times \rM_{X}(r, L)$. 

Inspired by earlier works in \cite{LN21, TT21}, define the Fourier-Mukai kernel $\cF$ over $X_{n} \times \rM_{X}(r, L)$ by taking the $S_{n}$-invariant part of the push-forward $q_{*}(\bigotimes_{i} q_{i}^{*}\sigma^{*}\cE)^{S_{n}}$, where $q : X^{n} \times \rM_{X}(r, L) \to X_{n} \times \rM_{X}(r, L)$ is the projection (Section \ref{ssec:kernel}). Then $\cF$ is a rank $r^{n}$ vector bundle over $X_{n} \times \rM_{X}(r, L)$. We consider the Fourier-Mukai transform $\Phi_{\cF} : \rD^{b}(X_{n}) \to \rD^{b}(\rM_{X}(r, L))$ and show that $\Phi_{\cF}$ is fully-faithful.

Applying the Bondal-Orlov criterion (Theorem \ref{thm:BondalOrlov}), the fully-faithfulness can be shown by evaluating cohomology groups of the form $\cF_{\bp} \otimes \cF_{\bq}^{*}$. Using the deformation of vector bundles, we may replace the problem by computation of cohomology of bundles of the form $\bigotimes_{i=1}S_{\lambda_{i}}\cE_{p_{i}}$, where $p_{i}$ are distinct points and $S_{\lambda_{i}}\cE_{p_{i}}$ are Schur functors of the bundle $\cE_{p_{i}}$ associated to partitions $\lambda_{i}$. 

The cohomology groups of these `standard' bundles can be evaluated by employing the Borel-Weil-Bott-Teleman theory (Section \ref{ssec:Teleman}, \cite{Tel98}) once the bundles are over the moduli stack $\cM_{X}(r, \cO)$ of all bundles with trivial determinants. Under the numerical condition $r > 2n$, we identify these cohomologies with that over $\rM_{X}(r, L)$ by studying the contribution of the unstable locus (Section \ref{sec:quantization}, \cite{HL15}) and geometry of moduli spaces of parabolic bundles (Section \ref{ssec:moduliparabolic}). 

\subsection{Some questions}

Here we leave some questions for the interested readers. 

We believe the large rank assumption ($r > 2n$) is not essential, but in our proof, it is required to eliminate the contribution of the unstable locus (Section \ref{ssec:quantizationmoduli}) and realize cohomological boundedness via the deformation argument. Our combinatorial approach does not seem to work for larger $r$ (Section \ref{ssec:cohomologyonstack}). 

\begin{question}\label{que:smallrankcase}
Can we lift the large rank assumption from Theorem \ref{thm:embedding}?
\end{question}

From a motivic computation and earlier results \cite{GL20, LM21, LM23}, it is expected that many copies of $\rD^{b}(X_{n})$ are embedded in $\rD^{b}(\rM_{X}(r, L))$ and these copies are obtained by twisting the image with the pluricanonical divisor. 

\begin{question}\label{que:semiorthogonaldecomposition}
Find an explicit semiorthogonal decomposition of $\rD^{b}(\rM_{X}(r, L))$. 
\end{question}

In our earlier work \cite{LM21, LM23}, the vanishing of cohomology was also used to show that $\cE_{p}$ is an arithmetically Cohen-Macaulay (ACM) bundle on $\rM_{X}(r, L)$ for any $p \in X$. 

\begin{question}\label{que:ACMbundle}
For any $\bp \subset X$ and a collection of partitions $\lambda_{1}, \lambda_{2}, \cdots, \lambda_{k}$, under which condition is the product of Schur functors $\bigotimes_{i=1}^{k}S_{\lambda_{i}}\cE_{p_{i}}$ an ACM bundle? Using this, can we show that $\rM_{X}(r, L)$ has nontrivial families of ACM bundles with arbitrary dimensions? In other words, is $\rM_{X}(r,L)$ of wild representation type \cite{CH11}?
\end{question}

A natural question that arises from Theorem \ref{thm:Fanovisitor} is the following. 

\begin{question}\label{que:abelian}
Is every abelian variety a Fano visitor?
\end{question}

It turns out that a parallel question for motives is true as follows. See Section \ref{sec:Fanovisitor} for details.

\begin{proposition}\label{motivicFanovisitor}
All symmetric products of curves and abelian varieties are motivic Fano visitors.
\end{proposition}

\subsection{Structure of the paper}

Section \ref{sec:parabolic} reviews the moduli space/stack of parabolic bundles, functorial morphisms between them, Schur functors of the universal bundle, and the GIT construction. All results are classical. Section \ref{sec:BondalOrlov} defines the Fourier-Mukai kernel F. Section \ref{sec:quantization} explains the negligibility of the contribution of unstable loci. In Section \ref{sec:boundedness}, employing the Borel-Weil-Bott-Teleman theory, we prove the boundedness and triviality of certain vector bundles, which is a necessary condition in the Bondal-Orlov criterion. Section \ref{sec:simple} shows the simpleness of the restricted Fourier-Mukai kernel and completes the proof of Theorem A. In the last section (Section \ref{sec:Fanovisitor}), we prove Theorem B and discuss the Fano visitor problem for motives.

\subsection*{Convention}

We work over $\CC$. We use $X$ to denote a smooth projective curve of genus $g \geq 2$ and $L$ is a line bundle of degree $d$ on $X$. 

\acknowledgement 

Part of this work was done when the first author was working at the Institute of the Mathematical Sciences of the Americas, University of Miami as an IMSA, Research Assistant Professor. He thanks Ludmil Katzarkov and Simons Foundation for partially supporting this work via Simons Investigator Award-HMS. He also thanks to Claire Voisin for helpful discussions and telling him about her work on Fano visitor problem for motives.

\section{Moduli space of parabolic bundles}\label{sec:parabolic}

In this section, we give an overview of the moduli space of parabolic bundles. 

\subsection{Moduli spaces of parabolic bundles}\label{ssec:moduliparabolic}

Let $X$ be a smooth projective curve of genus $g$ and let $\bp = (p_{1}, \cdots, p_{k})$ be an ordered set of distinct closed points on $X$. For notational simplicity, we only discuss parabolic bundles with full flags.

\begin{definition}\label{def:parabolicbundle}
A \emph{parabolic bundle} over $(X, \bp)$ is a collection of data $(E, \{W_{\bullet}^{i}\})$ where
\begin{enumerate}
\item $E$ is a rank $r$ vector bundle over $X$;
\item For each $1 \le i \le k$, $W_{\bullet}^{i} \in \Fl(E|_{p_{i}})$, in other words, $W_{\bullet}^{i}$ is a strictly increasing sequence of subspaces of $E|{p_i}$ as follows.
\[
	0 \subsetneq W_{1}^{i} \subsetneq W_{2}^{i} \subsetneq \cdots \subsetneq W_{r-1}^{i} \subsetneq E|_{p_{i}}
\]
\end{enumerate}
\end{definition}

\begin{definition}\label{def:parabolicmoduli}
Let $\cM_{X, \bp}(r, L)$ (resp. $\cM_{X, \bp}(r, d)$) be the moduli stack of rank $r$, determinant $L$ (resp. degree $d$) parabolic bundles over $(X, \bp)$.
\end{definition}

When $k = 0$, so there is no parabolic point, then $\cM_{X, \bp}(r, L)$ is the  moduli stack $\cM_{X}(r, L)$ of rank $r$ vector bundles. Let $\cE$ be the universal bundle over $X \times \cM_{X}(r, L)$. There is a forgetful morphism $\pi : \cM_{X, \bp}(r, L) \to \cM_{X}(r, L)$ and for each point $[E] \in \cM_{X}(r, L)$, its fiber $\pi^{-1}([E])$ is a product of flag varieties $\prod_{i=1}^{k}\Fl(E|_{p_{i}})$. Thus, 
\[
	\cM_{X, \bp}(r, L) = \times_{\cM_{X}(r, L)}\Fl(\cE|_{p_{i}}).
\]

To obtain a separated moduli stack with projective moduli space, one can employ a stability condition. For the moduli space of vector bundles, there is a standard notion of slope stability, but for parabolic bundles, the stability condition depends on numerical data, and they form a family. 

\begin{definition}\label{def:parabolicweight}
A \emph{parabolic weight} $\ba$ is a collection of data $\ba = (a_{\bullet}^{1}, \cdots, a_{\bullet}^{k})$ where each $a_{\bullet}^{i}$ is a length $r$ strictly decreasing sequence of real numbers 
\[
	1 > a_{1}^{i} > a_{2}^{i} > \cdots > a_{r-1}^{i} > a_{r}^{i} \ge 0.
\]
If $a_{r}^{i} = 0$ for all $i$, we say $\ba$ is \emph{normalized}.
\end{definition}

For a pointed curve $(X, \bp)$ with $|\bp| = k$, the space of normalized parabolic weights is the interior of $\Delta_{r-1}^{k}$, where $\Delta_{r-1}$ is an $(r-1)$-dimensional simplex. Indeed, $\Delta_{r-1} = \{(x_{i}) \in \RR^{r}\;|\; \sum_{i=1}^{r} x_{i} = 1, x_{i} \ge 0\}$ and $\mathrm{int}\;\Delta_{r-1}$ is identified with the set of parabolic weights on a single point, after setting $a_{0}^{i} = 1$, via $a_{\bullet}^{i} \mapsto (a_{j}^{i} - a_{j+1}^{i})_{0 \le j \le r-1} \in \Delta_{r-1}$. 

\begin{definition}\label{def:slope}
Let $(E, \{W_{\bullet}^{i}\})$ be a parabolic bundle. The \emph{parabolic degree} of $(E, \{W_{\bullet}^{i}\})$ with respect to $\ba$ is 
\[
	\pdeg_{\ba} (E, \{W_{\bullet}^{i}\}) := \deg E + \sum_{i=1}^{k}\sum_{j=1}^{r}a_{j}^{i}.
\]
Its \emph{parabolic slope} is 
\[
	\mu_{\ba} (E, \{W_{\bullet}^{i}\}) = \frac{\pdeg_{\ba} (E, \{W_{\bullet}^{i}\})}{\mathrm{rank}\; E}.
\]
\end{definition}

Fix a parabolic bundle $(E, \{W_{\bullet}^{i}\})$ over $(X, \bp)$. Let $F \subset E$ be a subbundle. For each point $p_{i}$, consider a (non-strictly increasing) filtration
\[
	W_{1}^{i} \cap F|_{p_{i}} \subset W_{2}^{i} \cap F|_{p_{i}} \subset \cdots \subset W_{r-1}^{i} \cap F|_{p_{i}}
\]
of $F|_{p_{i}}$. We define a full flag $W|_{F \bullet}^{i}$ of $F|_{p_{i}}$, by taking $(W|_{F}^{i})_{j}$ as $W_{\ell}^{i} \cap F_{p_{i}}$ with the smallest index $\ell$ such that $\dim (W_{\ell}^{i} \cap F_{p_{i}}) = j$. We also define the induced parabolic weight $\bfb = (\bfb_{\bullet}^{k})$ as $\bfb_{j}^{i} = \ba_{\ell}^{i}$, hence a (non-normalized) \emph{parabolic subbundle} $(F, \{W|_{F \bullet}^{i}\})$. By taking the quotient bundle $Q = E/F$ and the quotient filtration $\im (W_{j}^{i} \to Q|_{p_{i}})$, one can define a \emph{quotient parabolic bundle} $(Q, \{(W/F)_{\bullet}^{i}\})$ in a similar way. We define the quotient parabolic weight $\bc$ on $(Q, \{(W/F)_{\bullet}^{i}\})$ by taking the complementary weight data of $\bfb$. 

\begin{definition}\label{def:projectiveparabolicmoduli}
Fix a parabolic weight $\ba$. We say a parabolic bundle $(E, \{W_{\bullet}^{i}\})$ is \emph{$\ba$-(semi)-stable} if for any parabolic subbundle $(F, \{W|_{F \bullet}^{i}\})$ with induced parabolic weight $\bfb$, 
\[
	\mu_{\bfb}(F, \{W|_{F \bullet}^{i}\}) (\le) < \mu_{\ba}(E, \{W_{\bullet}^{i}\}). 
\]
Let $\cM_{X, \bp}(r, L, \ba) \subset \cM_{X, \bp}(r, L)$ be the substack of $\ba$-semistable parabolic bundles. There is a good moduli space $p : \cM_{X, \bp}(r, L, \ba) \to \rM_{X, \bp}(r, L, \ba)$, which is a normal projective variety. 
\end{definition}

If $\ba$ is general, then the stability and the semistability coincide. Then $\rM_{X, \bp}(r, L, \ba)$ is a smooth projective variety  \cite{MS80}. When $k = 0$, we denote the moduli stack of stable (resp. semistable) bundles by $\cM_{X}(r, L)^{s}$ (resp. $\cM_{X}(r, L)^{ss}$). 

\subsection{Functorial morphisms}\label{ssec:functorial}

Because of the connection between type A conformal blocks and the moduli stack of (untwisted) principal parabolic $\SL_{r}$-bundles \cite{BL94, Pau96, MY20, MY21}, the case of $L = \cO$ has been spelled out most explicitly in literature. For a general $L \in \Pic^{d}(X)$, we may describe $\rM_{X, \bp}(r, L, \ba)$ as a contraction of $\rM_{X, \bp'}(r, \cO, \ba')$ for some $\bp'$ and $\ba'$. In this section, we describe functorial morphisms between moduli stacks and moduli spaces.

Let $\bp := (p_{1}, \cdots, p_{k})$ and $\bp' := \bp \sqcup \{p_{k+1}\}$. Fix $d$ such that $1 \le d \le r-1$. For a parabolic bundle $(E, \{W_{\bullet}^{i}\}_{1 \le i \le k+1}) \in \cM_{X, \bp'}(r, \cO)$, consider the following epimorphism
\begin{equation}\label{eqn:fiberwisemodification}
	E \to E|_{p_{k+1}} \to E_{p_{k+1}}/W_{d}^{k+1} \to 0.
\end{equation}
Let $E_{d-r}$ be the kernel. Then $E_{d-r}$ is a vector bundle of the determinant $\cO(-(r-d)p_{k+1}) = \cO((d-r)p_{k+1})$. Forgetting all flags on $p_{k+1}$, we have an induced parabolic bundle $(E, \{W_{\bullet}^{i}\}_{1 \le i \le k})$ over $(X, \bp)$. The map $(E, \{W_{\bullet}^{i}\}_{1 \le i \le k+1}) \mapsto (E_{d-r}, \{W_{\bullet}^{i}\}_{1 \le i \le k})$ induces a morphism of stacks
\[
	m_{d} : \cM_{X, \bp'}(r, \cO) \to \cM_{X, \bp}(r, \cO((d-r)p_{k+1})).
\]

By selecting a stability appropriately, we may induce a morphism between their good moduli spaces. Let $\ba \in \Delta_{r-1}^{k}$ be a general parabolic weight. For the $(k+1)$-st point, we define $a_{\bullet}^{k+1}$ as $a_{j}^{k+1} < \epsilon$ for $j > d$ and $a_{j}^{k+1} > 1 - \epsilon$ for $j \le d$, for sufficiently small $\epsilon > 0$. We set $\ba' := \ba \cup a_{\bullet}^{k+1}$. By comparing the stabilities (Consult the proof of \cite[Proposition 2.9]{LM23}), we have an induced morphism of stacks
\[
	m_{d} : \cM_{X, \bp'}(r, \cO, \ba') \to \cM_{X, \bp}(r, \cO((d-r)p_{k+1}), \ba)
\]
and the corresponding morphism between their good moduli spaces (we use the same notation, if there is no chance of confusion)
\begin{equation}\label{eqn:modification}
	m_{d} : \rM_{X, \bp'}(r, \cO, \ba') \to \rM_{X, \bp}(r, \cO((d-r)p_{k+1}), \ba).
\end{equation}

On the other hand, by tensoring an appropriate line bundle $A$ on $E$, we obtain an isomorphism 
\begin{equation}\label{eqn:twistingiso}
\begin{split}
	\cM_{X, \bp}(r, L, \ba) &\cong \cM_{X, \bp}(r, L \otimes A^{r}, \ba),\\
	(E, \{W_{\bullet}^{i}\}) & \mapsto (E \otimes A, \{W_{\bullet}^{i}\}).
\end{split}
\end{equation}
Therefore, if $\deg (L_{1} \otimes L_{2}^{-1})$ is a multiple of $r$, $\cM_{X, \bp}(r, L_{1}) \cong \cM_{X, \bp}(r, L_{2})$ and there are similar isomorphisms between moduli stacks of $\ba$-semistable bundles and their good moduli spaces. By composing \eqref{eqn:twistingiso} and \eqref{eqn:modification}, if $\deg L = d$, we obtain a morphism $\rM_{X, \bp}(r, \cO, \ba') \to \rM_{X, \bp}(r, L, \ba)$, induced by $m_{d}$. 

Assume further that $(r, d) = 1$ and $\ba$ is sufficiently small, in the sense that $\sum_{i=1}^{k}\sum_{j=1}^{r-1}a_{j}^{i} < \epsilon$. Then the forgetful morphism $\pi : \cM_{X, \bp}(r, L) \to \cM_{X}(r, L)$ induces
\[
\begin{split}
	\pi : \cM_{X, \bp}(r, L, \ba) &\to \cM_{X}(r, L)^{s}\\
	(E, \{W_{\bullet}^{i}\}) &\mapsto E,
\end{split}
\]
and the fiber is a product of flag varieties, because the flag structure does not affect to the stability computation. 

In summary, if $\ba$ is sufficiently small, we have the following commutative diagram
\begin{equation}\label{eqn:commutativediagram}
\xymatrix{&\cM_{X, \bp'}(r, \cO) \ar[r]^{m_{d}} \ar@/_4.0pc/[ddd]_{\pi}& \cM_{X, \bp}(r, L)\ar@/^4.0pc/[ddd]^{\pi}\\
\rM_{X, \bp'}(r, \cO, \ba') \ar@{-->}[d]^{\pi}&\cM_{X, \bp'}(r, \cO, \ba') \ar[l]_{p}\ar[u]_{j} \ar[r]^{m_{d}} \ar@{-->}[d]^{\pi}& \cM_{X, \bp}(r, L, \ba) \ar[u]_{j} \ar[d]^{\pi} \ar[r]^{p} & \rM_{X, \bp}(r, L, \ba) \ar[d]^{\pi}\\
\rM_{X}(r, \cO) &\cM_{X}(r, \cO)^{ss} \ar[l]_{p} \ar[d]^{j} & \cM_{X}(r, L)^{s} \ar[r]^{p} \ar[d]^{j}& \rM_{X}(r, L)\\
& \cM_{X}(r, \cO) & \cM_{X}(r, L).
}
\end{equation}
Each $\pi$ is a forgetful map, $p$ is a good moduli map, and $j$ is a natural inclusion. Note that some $\pi$ on the left-hand side are not regular morphisms but are defined over an open substack. 

\subsection{Schur functors and Borel-Weil-Bott theorem}\label{ssec:Schur}

Let $F$ be a rank $r$ vector bundle on a stack $\cM$. For a partition $\lambda = (\lambda_{1} \ge \lambda_{2} \ge \cdots \ge \lambda_{k} > 0)$ of $n$ with at most $r$ parts, we denote the associate Schur functor bundle by $S_{\lambda}F$. For instance, if $\lambda = (n)$, $S_{\lambda}F = \mathrm{Sym}^{n}F$. If $\lambda = (1, 1, \cdots, 1)$, $S_{\lambda}F = \wedge^{n}F$. 

Equivalently, following the standard representation theory of $\mathfrak{sl}_{r}$, any partition $\lambda$ can be understood as a sum of dominant integral weights $\lambda = \sum a_{j}\omega_{j}$. Here $\omega_{j}$ is the $j$-th fundamental weight. Then $a_{j} = \lambda_{j} - \lambda_{j+1}$.

Let $\cE$ be the universal bundle over $\cM_{X}(r, L)$. Note that there is a forgetful $\cM_{X}(r, L)$-morphism $\cM_{X, \bp}(r, L) \to \Fl(\cE|_{p_{i}})$. For each partition $\lambda$ with at most $r$ parts, there is an associated line bundle $L_{p_{i}, \lambda}$ over $\Fl(\cE|_{p_{i}})$. By taking the pull-back, we have a line bundle $L_{p_{i}, \lambda}$ over $\cM_{X, \bp}(r, L)$. By the Borel-Weil-Bott theorem, we have 
\begin{equation}\label{eqn:pushforwardformula}
	\pi_{*}L_{p_{i}, \lambda} = S_{\lambda}\cE_{p_{i}}.
\end{equation}
Furthermore, using the Leray spectral sequence, one can identify their cohomology groups:
\begin{equation}\label{eqn:cohomologycomparison}
	\rH^{*}(\cM_{X, \bp}(r, L), L_{p_{i}, \lambda}) \cong \rH^{*}(\cM_{X}(r, L), S_{\lambda}\cE_{p_{i}}).
\end{equation}
When $\ba$ is small, since the restricted map $\pi : \cM_{X, \bp}(r, L, \ba) \to \cM_{X}(r, L)^{s}$ is also a fibration with the same fiber because $\cM_{X, \bp}(r, L, \ba) = \cM_{X}(r, L)^{s} \times_{\cM_{X}(r, L)}\cM_{X, \bp}(r, L)$. So the same cohomology formula is true:
\begin{equation}\label{eqn:cohomologycomparison2}
	\rH^{*}(\cM_{X, \bp}(r, L, \ba), L_{p_{i}, \lambda}) \cong \rH^{*}(\cM_{X}(r, L)^{s}, S_{\lambda}\cE_{p_{i}}).
\end{equation}

On the other hand, note that the morphism $m_{d} : \cM_{X, \bp'}(r, \cO) \to \cM_{X, \bp}(r, L)$ in \eqref{eqn:fiberwisemodification} does not make any change along $p_{i}$ for $1 \le i \le k$. Thus, we have 
\begin{equation}\label{eqn:pullbackoflambda}
	m_{d}^{*}L_{p_{i}, \lambda} = L_{p_{i}, \lambda}. 
\end{equation}
From now on, for notational simplicity, we will suppress the pull-back $m_{d}^{*}$ if there is no chance of confusion.

\subsection{GIT construction}\label{ssec:GITconstruction}

The moduli space $\rM_{X, \bp}(r, L, \ba)$ of semistable parabolic bundles can also be constructed by GIT. In this section, we briefly describe a construction. 

Fix a degree one line bundle $\cO_{X}(1)$ over $X$. Take a sufficiently large $m \in \ZZ$ so that $\rH^{1}(E(m)) = 0$ and $E(m)$ is globally generated for all $(E, \{W_{\bullet}^{i}\}) \in \cM_{X, \bp}(r, L, \ba)$. Let $\chi_{m} := \rH^{0}(E(m)) = d + r(m + 1-g)$. Let $\bQ(m) = \mathrm{Quot}(\cO_{X}^{\chi_{m}})$ be the quot scheme parametrizing quotients $\cO_{X}^{\chi_{m}} \to F \to 0$ such that the Hilbert polynomial of $F$ is that of $E(m)$. Let $\bR(m) \subset \bQ(m)$ be the locally closed subscheme parametrizing the quotients $\cO_{X}^{\chi_{m}} \stackrel{\varphi}{\to} F \to 0$ such that $\rH^{1}(F) = 0$, $\rH^{0}(\cO_{X}^{\chi_{m}}) \stackrel{\wedge^{r}\varphi}{\cong} \rH^{0}(F) \cong \CC^{\chi_{m}}$, $F$ is locally free, and $\wedge^{r}F \cong L(rm)$. 

For the universal quotient $\cO_{X \times \bR(m)}^{\chi_{m}} \to \cF \to 0$, let 
\[
	\widetilde{\bR}(m) := \times_{\bR(m)}\Fl(\cF|_{p_{i}})
\]
be the fiber product of full-flag bundles. Then $\widetilde{\bR}(m)$ admits an $\SL_{\chi_{m}}$-action, and $\rM_{X, \bp}(r, L, \ba)$ is constructed as a GIT quotient $\widetilde{\bR}(m)\git \SL_{\chi_{m}}$ with a certain linearization. 

The linearization is constructed explicitly in \cite{Bho89}. Let $Z := \PP \rHom(\wedge^{r}\CC^{\chi_{m}}, \rH^{0}(L(rm)))^{*}$. Then for any $[\cO_{X}^{\chi_{m}} \stackrel{\varphi}{\to} F \to 0] \in \bR(m)$, we assign 
\[
	\wedge^{r}\CC^{\chi_{m}} \stackrel{\wedge^{r}\varphi}{\to}\wedge^{r}\rH^{0}(F) \to \rH^{0}(\wedge^{r}F) \cong \rH^{0}(L(rm)).
\]
This assignment induces a morphism $\bR(m) \to Z$, which is indeed an embedding \cite[Section 7]{Tha96}. Furthermore, for each $p_{i} \in \bp$, we have an evaluation map $\psi_{i} : \CC^{\chi_{m}} \cong \rH^{0}(F) \to F|_{p_{i}}$. Then for each $W_{j}^{i}$, we have $\psi_{i}^{-1}(W_{j}^{i}) \in \Gr(\chi_{m} - r + j, \CC^{\chi_{m}})$. Therefore, we have a morphism 
\begin{equation}\label{eqn:GITembedding}
	\widetilde{\bR}(m) \to Z \times \prod_{i=1}^{k}\prod_{j=1}^{r-1}\Gr(\chi_{m}-r+j, \CC^{\chi_{m}}).
\end{equation}
In \cite{Bho89}, Bhosle described an explicit ample line bundle $A(\ba)$ on the right-hand side, such that $\widetilde{\bR}(m) \git_{A(\ba)} \SL_{\chi_{m}} \cong \rM_{X, \bp}(r, L, \ba)$. On the atlas $\widetilde{\bR}(m) \to [\widetilde{\bR}(m)/\SL_{\chi_{m}}]$, note that the line bundle $L_{p_{i}, \omega_{j}}$ is given by the restriction of $\cO_{\Gr(\chi_{m} - r + j, \CC^{\chi_{m}})}(1)$ for the $i$-th factor in \eqref{eqn:GITembedding}.

A parabolic bundle $(E, \{W_{\bullet}^{i}\})$ is $\ba$-unstable if it has a maximal destabilizing parabolic subbundle $(F, \{W|_{F \bullet^{i}}\})$ such that $\mu_{\bfb}(F, \{W|_{F \bullet}^{i}\}) > \mu_{\ba}(E, \{W_{\bullet}^{i}\})$. If we set $s = \mathrm{rank}\; F$, $e = \deg F$, and $J^{i} \subset [r]$ such that 
\[
	\mu_{\bfb}(F, \{W|_{F \bullet}^{i}\}) = \frac{e + \sum_{i=1}^{k}\sum_{j \in J^{i}}a_{j}^{i}}{s},
\]
this irreducible component is indexed by the numerical triple $(s, e, \{J^{i}\})$. 

\begin{definition}\label{def:unstablestratum}
We denote the unstable stratum associated to the numerical data $(s, e, \{J^{i}\})$ by $S_{(s, e, \{J^{i}\})} \subset \widetilde{\bR}(m)$.
\end{definition}

Under the condition $L = \cO$, the codimension of the unstable locus is evaluated by Sun \cite{Sun00}. See \cite[Section 3.2]{MY20} for a summary. In particular, in \cite[p.254, (3.2)]{MY20}, it was shown that codimension of $S_{(s, e, \{J^{i}\})}$ is given by 
\[
	s(r-s)(g-1) + \sum_{i=1}^{k}\mathrm{codim}\; Y^{i} + re, 
\]
where $Y^{i} \subset \Fl(\CC^{r})$ is a certain flag variety. In \cite[Lemma 5.2]{Sun00} (see also \cite[p.254]{MY20}), it was shown that $\sum_{i=1}^{k}\mathrm{codim}\; Y^{i} + re$ is positive. In summary, we have:

\begin{lemma}\label{lem:unstablecodim}
The codimension of the unstable locus $S_{(s, e, \{J^{i}\})}$ is at least $s(r-s)(g-1) + 1$. 
\end{lemma}

%%%%%%%%%%%%%%%%%%%%%%%%%%%%%%%%%%%%%

\section{Bondal-Orlov criterion and cohomology of line bundles}\label{sec:BondalOrlov}

Many natural functors between two derived categories of algebraic varieties are constructed as Fourier-Mukai transforms $\Phi_{\cF}$. In this section, we describe the Fourier-Mukai kernel that we will use and reformulate the Bondal-Orlov criterion for the fully-faithfulness of $\Phi_{\cF}$.

\subsection{Fourier-Mukai kernel}\label{ssec:kernel}

Let $X^{n}$ be the product of $n$ copies of $X$ and let $q_{i} : X^{n} \to X$ be the projection to its $i$-th factor. There is a natural $S_{n}$-action on $X^{n}$. We denote the quotient map $q : X^{n} \to X_{n} = X^{n}/S_{n} \cong \mathrm{Hilb}^{n}(X)$, which is a finite flat morphism. This construction can be relativized, obviously. For any scheme or stack $\cM$, we have a quotient map $q_{\cM} : X^{n} \times \cM \to X_{n} \times \cM$. If there is no chance of confusion, we will suppress the subscript $\cM$ and denote it by $q$.

Let $\cE$ be the universal bundle on the moduli stack $\cM_{X}(r, L)$. Let $\cM_{X}(r, L)^{ss} \subset \cM_{X}(r, L)$ be the open substack of the semistable bundles and $p : \cM_{X}(r, L)^{ss} \to \rM_{X}(r, L)$ be the good moduli space morphism. 

Suppose that $(r, d) = 1$, where $d = \deg L$. Then there is a section $\sigma : \rM_{X}(r, L) \to \cM_{X}(r, L)^{ss} = \cM_{X}(r, L)^{s}$. A \emph{Poincar\'e bundle} on the coarse moduli space is $\sigma^{*}\cE$, the pull-back of the universal bundle over the moduli stack. Note that it depends on the choice of a section $\sigma$, which is equivalent to a choice of a line bundle on $\rM_{X}(r, L)$. 

\begin{definition}\label{def:FMkernel}
Consider a Poincar\'e bundle $\sigma^{*}\cE$ over $X \times \rM_{X}(r, L)$. We set 
\[
	(\sigma^{*}\cE)^{\otimes n} := \bigotimes_{i=1}^{n}q_{i}^{*}\sigma^{*}\cE.
\]
Then $(\sigma^{*}\cE)^{\otimes n}$ is an $S_{n}$-equivariant bundle of rank $r^{n}$. By taking the invariant functor, we obtain a vector bundle 
\begin{equation}\label{eqn:kernelF}
	\cF := (q_{*}(\sigma^{*}\cE)^{\otimes n})^{S_{n}}
\end{equation}
on $X_{n} \times \rM_{X}(r, L)$ \cite[Lemma 2.1]{TT21}.
\end{definition}

\begin{definition}\label{def:restrictedkernel}
For $\bp  = (\sum p_{j})\in X_{n}$, let $\cF_{\bp} = \iota_{\bp}^{*}\cF$ be the restriction of $\cF$ by $\iota_{\bp} : \{\bp\} \times \rM_{X}(r, L) \hookrightarrow X_{n}\times \rM_{X}(r, L)$. 
\end{definition}

\begin{lemma}[\protect{\cite[Corollary 2.8]{TT21}}]\label{lem:FMkernel} The restricted bundle $\cF_{\bp}$ is a deformation of $\bigotimes_{i}q_{i}^{*}\sigma^{*}\cE_{p_{i}}$. In other words, there is a family of bundles over $\AA^{1} \times \rM_{X}(r, L)$ such that the restriction to $\{0\} \times \rM_{X}(r, L)$ is isomorphic to $\bigotimes_{i}q_{i}^{*}\sigma^{*}\cE_{p_{i}}$ and that to $\{t\} \times \rM_{X}(r, L)$ for $t \ne 0$ is isomorphic to $\cF_{\bp}$.
\end{lemma}

\begin{definition}\label{def:FMfunctor}
Let $\Phi_{\cF} : \rD^{b}(X_{n}) \to \rD^{b}(\rM_{X}(r, L))$ be the Fourier-Mukai transform with the kernel $\cF$ in \eqref{eqn:kernelF}, that is, 
\[
	\Phi_{\cF}(E^{\bullet}) = Rp_{2 *}(Lp_{1}^{*}E^{\bullet} \otimes^{L} \cF).
\]
\end{definition}

\begin{remark}
The definition of the Fourier-Mukai kernel $\cF$ and the Fourier-Mukai transform $\Phi_{\cF}$ depend on a choice of Poincar\'e bundle $\sigma^{*}\cE$. To be more precise, we denote $\cF^{\sigma} := (q_{*}(\sigma^{*}\cE)^{\otimes n})^{S_{n}}$. Then for two sections $\sigma, \sigma' : \rM_{X}(r, L) \to \cM_{X}(r, L)^{s}$, ${\sigma'}^{*}\cE = (\sigma^{*}\cE) \otimes p_{2}^{*}A$ for some $A \in \Pic(\rM_{X}(r, L))$. Then 
\[
	\cF^{\sigma'} = (q_{*}({\sigma'}^{*}\cE)^{\otimes n})^{S_{n}} \cong (q_{*}(\sigma^{*}\cE \otimes p_{2}^{*}A)^{\otimes n})^{S_{n}} \cong (q_{*}(\sigma^{*}\cE)^{\otimes n})^{S_{n}} \otimes p_{2}^{*}A^{n} = \cF^{\sigma}\otimes p_{2}^{*}A^{n}.
\]
By the projection formula, 
\[
	\Phi_{\cF^{\sigma'}}(E^{\bullet}) = \Phi_{\cF^{\sigma}}(E^{\bullet}) \otimes A^{n}.
\]
Therefore, the fully-faithfulness does not change and we may choose any Poincar\'e bundle.
\end{remark}

We will prove Theorem \ref{thm:embedding} by showing that the functor $\Phi_{\cF} : \rD^{b}(X_{n}) \to \rD^{b}(\rM_{X}(r, L))$ is fully-faithful. 

\subsection{Bondal-Orlov criterion}

We show the fully-faithfulness by employing the classical result of Bondal and Orlov \cite[Theorem 1.1]{BO95}. We state a version adapted to our situation.

\begin{theorem}[Bondal-Orlov criterion]\label{thm:BondalOrlov}
The functor $\Phi_{\cF}$ is fully faithful if and only if the following three conditions hold:
\begin{enumerate}
\item (Simplicity) $\rH^{0}(\rM_{X}(r, L), \cF_{\bp} \otimes \cF_{\bp}^{*}) \cong \CC$.
\item (Cohomological boundedness) $\rH^{i}(\rM_{X}(r, L), \cF_{\bp} \otimes \cF_{\bp}^{*}) = 0$ for $i > n$. 
\item (Cohomological triviality) $\rH^{i}(\rM_{X}(r, L), \cF_{\bp} \otimes \cF_{\bq}^{*}) = 0$ for all $\bp \ne \bq \in X_{n}$ and $i \in \ZZ$.
\end{enumerate}
\end{theorem}

\begin{remark}\label{rmk:boundedandtriviality}
We say a coherent sheaf $F$ on an algebraic stack (or a scheme) $\cM$ is \emph{cohomologically bounded up to degree $n$} if $\rH^{i}(\cM, F) = 0$ for all $i > n$. $F$ is \emph{cohomologically trivial} if $\rH^{i}(\cM, F) = 0$ for all $i \ge 0$. 
\end{remark}

%%%%%%%%%%%%%%%%%%%%%%%%%%%%%%%%%%%%%

\section{Quantization}\label{sec:quantization}

The cohomological boundedness/triviality on the coarse moduli space will be shown by comparing the cohomology groups with those on the moduli stack. Then, we need to estimate the effect of unstable strata. By employing \cite{HL15}, we show that the contribution of the unstable loci is negligible if $n$ is relatively small.

\subsection{Contribution of unstable loci}\label{ssec:unstableloci}

Let $V$ be a smooth quasi-projective variety with a reductive group $G$ action. Let $A$ be a $G$-linearization on $V$. We review \cite{HL15} which explains how to compare $\rD^{b}([V/G])$ and $\rD^{b}([V^{ss}(A)/G])$ where the latter quotient stack $[V^{ss}(A)/G]$ has a good moduli space $V\git_{A}G$ (thus $\pi^{*} : \rD^{b}(V\git_{A}G) \to \rD^{b}([V^{ss}(A)/G])$ is fully-faithful). 

The Kempf-Ness stratification of $V$ can be constructed as the following. For each one-parameter subgroup $\lambda : \CC^{*} \to G$, let $Z \subset V^{\lambda}$ be an irreducible component of the torus fixed locus $X^{\lambda}$. Then, one may compute the numerical invariant 
\[
	\mu(\lambda, Z) := -\frac{\mathrm{wt}_{\lambda}A|_{Z}}{|\lambda|}.
\]
Take a pair $(\lambda, Z)$ such that $\mu(\lambda, Z)$ is the largest positive one. We set $Y_{\lambda, Z} := \{x \in V\;|\; \lim_{t \to 0}\lambda(t)\cdot x \in Z\}$ and $S_{\lambda, Z} := G\cdot Y_{\lambda, Z}$. Then $S_{\lambda, Z}$ is a stratum in the unstable locus. One can continue this construction by starting with $V \setminus S_{\lambda, Z}$, and the $G$-action on it. 

\begin{theorem}[Quantization Theorem \protect{\cite[Theorem 3.29]{HL15}}]\label{thm:quantization}
Let $\eta$ be the $\lambda$-weight of $\wedge^{\mathrm{top}}N_{S_{\lambda, Z}/V}^{*}|_{Z}$. Let $E^{\bullet} \in \rD^{b}([V/G])$, and suppose that the $\lambda$-weight of $\mathcal{H}^{*}(E^{\bullet}|_{Z})$ is supported on $(-\infty, \eta)$. Then 
\[
	\rH^{*}([V/G], E^{\bullet}) \cong \rH^{*}([V^{ss}(A)/G], E^{\bullet}|_{[V^{ss}(A)/G]}).
\]
\end{theorem}

Thus, if the given sheaf does not have a too large $\lambda$-weight, then its cohomology on the semistable locus coincides with the cohomology over the whole quotient stack. By induction, this theorem is valid for the GIT quotient with many components in the unstable loci.

\subsection{The case of moduli space of parabolic bundles}\label{ssec:quantizationmoduli}

We apply Theorem \ref{thm:quantization} to $\cM_{X, \bp}(r, \cO, \ba)$. The main result is Corollary \ref{cor:unstablelociarenegligible}, which shows that the unstable loci are negligible in the cohomology calculation.

Let $S = S_{(s, e, \{J^{i}\})}$ be a stratum of the unstable locus, and let the associated one-parameter subgroup be $\lambda(t)$. The $\lambda$-fixed locus $Z \subset S$ parametrizes data
\[
	\{[\cO^{\chi^{+}} \oplus \cO^{\chi^{-}} \stackrel{\varphi}{\to}
	E^{+}(m) \oplus E^{-}(m) \to 0], \{W_{\bullet}^{i}\}\}, 
\]
where $\varphi = \varphi^{+}\oplus \varphi^{-}$, $E^{+}$ (resp. $E^{-}$) is of degree $e$ (resp. $-e$), and $W_{\bullet}^{i} \subset E^{+} \cup E^{-}$, in other words, $W_{j}^{i} = (E^{+}|_{p_{i}}\cap W_{j}^{i}) \oplus (E^{-}|_{p_{i}} \cap W_{j}^{i})$. Here $\chi^{+}(m) = \dim \rH^{0}(E^{+}(m))$ and $\chi^{-}(m) = \dim \rH^{0}(E^{-}(m))$. 

If we take a general point of $S$, then both $E^{+}$ and $E^{-}$ are simple; hence $\lambda(t)$ acts on each factor as a scalar multiplication. So, up to normalization, $\lambda(t)$ acts as 
\[
	\left(\begin{array}{cc}t^{-\chi^{-}} & 0 \\ 0 & t^{\chi^{+}}\end{array}\right).
\]
For a general point in $Z$, $(E^{+}, \{W|_{E^{+} \bullet}^{i}\})$ has the following property. For any $j$ with $J_{k}^{i} \le j < J_{k+1}^{i}$, $\dim E^{+}|_{p_{i}} \cap W_{j}^{i} = k$ and $\dim E^{-}|_{p_{i}}  \cap W_{j}^{i} = j-k$. Therefore, 
\[
	\dim \CC^{\chi^{+}} \cap \psi_{i}^{-1}(W_{j}^{i}) = \chi^{+} - s + k,
\]
and 
\[
	\dim \CC^{\chi^{-}} \cap \psi_{i}^{-1}(W_{j}^{i}) = \chi^{-} - (r - s) + j - k.
\]
Thus, 
\[
\begin{split}
	\mathrm{wt}_{\lambda}L_{p_{i}, \omega_{j}} &= -\chi^{-}(\chi^{+}-s + k) + \chi^{+}(\chi^{-}-r+s+j-k) = \chi(s-k) - \chi^{+}(r-j)\\
	&= r(m+1-g)(s-k) - (e+s(m+1-g))(r-j)\\
	&= (m+1-g)(sj - rk) - e(r-j).
\end{split}
\]
Since $j -k \le \dim E^{-}|_{p_{i}} = r-s$, $sj - rk \le s(r-s+k) - rk = (r-s)(s-k)$. And if $m \gg 0$, because $\chi$ is a linear polynomial for $m$ but $e$ and  $j$ are bounded constants, 
\[
\begin{split}
	\mathrm{wt}_{\lambda}L_{p_{i}, \omega_{j}} &=
	(m+1-g)(sj-rk) - e(r-j) \le (m+1-g)(r-s)(s-k) -e(r-j)\\
	&= \chi\frac{(r-s)(s-k)}{r} - e(r-j) < \chi \frac{(r-s)(s-k)+1}{r}
	< \chi\frac{(r-s)s+1}{r}.
\end{split}
\]
Therefore, for a dominant weight $\sum a_{j}\omega_{j}$ with $a_{j} \ge 0$, the $\lambda$-weight of the associated line bundle $L_{p_{i}, \lambda} = L_{p_{i}, \sum a_{j}\omega_{j}}$ is at most 
\[
	\chi (\sum a_{j})\frac{(r-s)s+1}{r}.
\]

On the other hand, by \cite[Section 7]{Tha96}, $\lambda(t)$ acts on $N_{S/V}|_{Z}$ by a multiplication of $t^{-\chi}$. Thus, 
\[
	\eta = \mathrm{wt}_{\lambda}\wedge^{\mathrm{top}}N_{S/V}^{*}|_{Z} = \chi \cdot \mathrm{rank}\; N_{S/V} \ge \chi \left(s(r-s)(g-1) + 1\right)
\]
by Lemma \ref{lem:unstablecodim}. In particular, if $\sum a_{j} \le r((r-s)s(g-1)+1)/((r-s)s+1)$, 
\[
	\mathrm{wt}_{\lambda} L_{p_{i}, \sum a_{j}\omega_{j}} < \chi (\sum a_{j})\frac{(r-s)s+1}{r} \le \chi ((r-s)s(g-1)+1) \le \eta.
\]
The minimum of $r((r-s)s(g-1)+1)/((r-s)s+1)$ for $1 \le s \le r-1$ is achieved when $s = 1$, hence it is $r((r-1)(g-1)+1)/(r-1+1) = (r-1)(g-1)+1$. Then we obtain the following result by the Quantization Theorem (Theorem \ref{thm:quantization}). Note that we use $\lambda$ to describe a partition, not a one-parameter subgroup in the statement.

\begin{proposition}\label{prop:unstablelociarenegligible}
For a partition $\lambda = \sum a_{j}\omega_{j}$ with $\sum a_{j} \le (r-1)(g-1)+1$,  
\[
	\rH^{*}([\widetilde{\bR}(m)/\SL_{\chi_{m}}],  L_{p_{i}, \lambda}) \cong \rH^{*}([\widetilde{\bR}(m)^{ss}(A(\ba))/\SL_{\chi_{m}}],  L_{p_{i}, \lambda}) \cong \rH^{*}(\cM_{X}(r, \cO, \ba), L_{p_{i}, \lambda}).
\]
\end{proposition}

\begin{corollary}\label{cor:unstablelociarenegligible}
For a partition $\lambda = \sum a_{j}\omega_{j}$ with $\sum a_{j} \le (r-1)(g-1)+1$, 
\begin{equation}\label{eqn:unstablelociarenegligible}
	\rH^{*}(\cM_{X, \bp}(r, \cO),  L_{p_{i}, \lambda}) \cong \rH^{*}(\cM_{X, \bp}(r, \cO, \ba), L_{p_{i}, \lambda}).
\end{equation}
\end{corollary}

\begin{proof}
There is an open embedding 
\[
	[\widetilde{\bR}(m)/\SL_{\chi_{m}}] \subset [\widetilde{\bR}(m+1)/\SL_{\chi_{m+1}}], 
\]
and we have an isomorphism of stacks
\[
	\cM_{X, \bp}(r, \cO) \cong \lim_{m \rightarrow}[\widetilde{\bR}(m)/\SL_{\chi_{m}}].
\]
For $L_{p_{i}, \lambda}$ with $p_{i} \in \bp$, we have a morphism 
\begin{equation}\label{eqn:inverselimit}
	\rH^{*}(\cM_{X, \bp}(r, \cO), L_{p_{i}, \lambda}) \to \lim_{m \leftarrow}\rH^{*}([\widetilde{\bR}(m)/\SL_{\chi_{m}}], \iota_{m}^{*}L_{p_{i}, \lambda})
\end{equation}
where $\iota_{m} : [\widetilde{\bR}(m)/\SL_{\chi_{m}}] \subset \cM_{X, \bp}(r, \cO)$. By Proposition \ref{prop:unstablelociarenegligible}, each cohomology is identified with the cohomology of a line bundle on a projective variety $\rM_{X, \bp}(r, \cO, \ba)$. Hence it is finite-dimensional. Thus, the inverse system on the right hand side satisfies the Mittag-Leffler condition. Therefore, the map in \eqref{eqn:inverselimit} is an isomorphism. 
\end{proof}

%%%%%%%%%%%%%%%%%%%%%%%%%%%%%%%%%%%%%

\section{Cohomological boundedness}\label{sec:boundedness}

The classical Borel-Weil-Bott theorem provides a recipe to compute the cohomology of all line bundles over the full-flag variety $\Fl(V)$ of a finite-dimensional vector space $V$. This is extended to the case of the moduli stack of vector bundles with trivial determinant by Teleman \cite{Tel98}. In this section, we review the Borel-Weil-Bott-Teleman theory and its implication to the cohomology boundedness/triviality. 

\subsection{Borel-Weil-Bott for curves}\label{ssec:Teleman}

Recall that $\Pic(\cM_{X}(r, \cO)) \cong \ZZ$. Let $\Theta \in \Pic(\cM_{X}(r, \cO))$ be the ample generator. For the universal family $\cE$ over $\cM_{X}(r, \cO)$, a point $p \in X$, and a partition $\lambda \vdash n$ of length at most $r-1$, let $S_{\lambda}\cE_{p}$ be the Schur functor applied to $\cE_{p}$ (Section \ref{ssec:Schur}). If we have $k$ distinct points $\bp = (p_{1}, p_{2}, \cdots, p_{k})$ and $n$ partitions $\lambda_{1}, \lambda_{2}, \cdots, \lambda_{k}$, we may construct a vector bundle 
\[
	\Theta^{h} \otimes \bigotimes_{i=1}^{k}S_{\lambda_{i}}\cE_{p_{i}}.
\]
Teleman's extension of the Borel-Weil-Bott theorem evaluates the cohomology groups of these bundles. 

Here we give some relevant representation theoretic definitions, specialized to $\SL_{r}$. Let $h$ be a fixed nonnegative integer. Let $\mathfrak{h}$ be the Cartan subalgebra of $\mathfrak{sl}_{r}$. On the Euclidean space $\mathfrak{h}^{*}$ with the normalized Killing form $(- ,-)$, let $\{\beta_{j}\}$ be the set of fixed simple roots, and $\{\omega_{j}\}$ be the associated fundamental weights. With respect to the Killing form, $\{\omega_{j}\}$ is the dual basis of $\{\beta_{j}\}$. We denote by $\rho$ the half sum of all positive roots, or equivalently, the sum of all fundamental weights. The set of hyperplanes $\{\lambda \;|\; (\lambda, \beta) \in (h+r)\ZZ\}$ where $\beta$ is a root of $\mathfrak{h}$, divides $\mathfrak{h}^{*}$ into polyhedral chambers, the so-called \emph{Weyl alcoves}. The alcove containing the small highest weights is called the \emph{positive alcove}. The positive alcove is an open simplex bounded by $\{(\lambda, \beta_{j}) > 0\}_{1 \le j \le r-1}$ and $(\lambda, \sum_{j=1}^{r-1}\beta_{j}) < h+r$. 

We say a weight $\lambda$ is \emph{regular} if $\lambda+\rho$ is on the interior of one of the alcoves. Otherwise, $\lambda$ is called \emph{singular}. For a regular weight, the length $\ell(\lambda)$ is defined as the number of Weyl reflections that requires to map $\lambda+\rho$ to $\mu+\rho$ in the positive alcove. Equivalently, $\ell(\lambda)$ is the minimum number of hyperplanes required to cross to move from the alcove containing $\lambda+\rho$ to the positive alcove. In this situation, $\mu$ is called the \emph{ground form} of $\lambda$. Since $\rho = \sum \omega_{j}$ and $\{\omega_{j}\}$ is the dual basis of $\{\beta_{j}\}$, for a given weight $\lambda = \sum a_{j}\omega_{j}$, $\lambda + \rho$ is in the positive alcove if and only if $a_{j} \ge 0$ and $\sum a_{j} \le h$. 

\begin{remark}\label{rem:positivealcovewhenhiszero}
If $h = 0$, the only $\lambda$ such that $\lambda + \rho$ is in the positive alcove is $\lambda = 0$. 
\end{remark}

\begin{theorem}[\protect{\cite{Tel98}}]\label{thm:BWB}
Fix $h \ge 0$. Suppose that, with respect to $h + r$, all $\lambda_{i}$'s are regular. Then $\rH^{\ell}(\cM_{X}(r, \cO), \Theta^{h} \otimes \bigotimes_{i=1}^{k}S_{\lambda_{i}}\cE_{p_{i}}) \cong \rH^{0}(\cM_{X}(r, \cO), \Theta^{h} \otimes \bigotimes_{i=1}^{k}S_{\mu_{i}}\cE_{p_{i}})$, where $\ell = \sum \ell(\lambda_{i})$ is the sum of the lengths of $\lambda_{i}$'s and $\mu_{i}$ is the ground form of $\lambda_{i}$, and all other cohomology groups are trivial. If one of $\lambda_{i}$'s is singular, then $\rH^{*}(\cM_{X}(r, \cO), \Theta^{h} \otimes \bigotimes_{i=1}^{k}S_{\lambda_{i}}\cE_{p_{i}}) = 0$.
\end{theorem}

If we set $\bp = (p_{1}, p_{2}, \cdots, p_{k})$ as the set of parabolic points, the moduli stack $\cM_{X, \bp}(r, L)$ of the parabolic bundles over $(X, \bp)$ is obtained by taking the fiber product of the relative flag bundles. (Section \ref{ssec:moduliparabolic}) By the Borel-Weil-Bott and the degeneration of the Leray spectral sequence and the fact that the push-forward of $L_{p_{i}, \lambda_{i}}$ is $S_{\lambda_{i}}\cE_{p_{i}}$, we obtain 
\begin{equation}\label{eqn:cohidentification}
	\rH^{*}(\cM_{X, \bp}(r, \cO), \Theta^{h}\otimes \bigotimes_{i=1}^{k}L_{p_{i}, \lambda_{i}}) \cong
	\rH^{*}(\cM_{X}(r, \cO), \Theta^{h} \otimes \bigotimes_{i=1}^{k}S_{\lambda_{i}}\cE_{p_{i}}).
\end{equation}
Therefore, Teleman's theorem can be understood as the cohomology evaluation of line bundles on the moduli stack of parabolic bundles.

\subsection{Cohomological boundedness and triviality on the moduli stack}\label{ssec:cohomologyonstack}

In this section, all bundles are over $\cM_{X}(r, \cO)$. We start with a few combinatorial lemmas. Let $\cE$ be the universal bundle over $X \times \cM_{X}(r, \cO)$. Recall that a coherent sheaf $F$ is cohomologically bounded up to degree $n$ if $\rH^{i}(F) = 0$ for all $i > n$ (Remark \ref{rmk:boundedandtriviality}). 

Under the assumption $h = 0$, here we describe the Weyl reflection more explicitly. Let $\lambda = \sum a_{j}\omega_{j}$ be a dominant weight. Then $\lambda + \rho = \sum (a_{j}+1)\omega_{j}$. We set $h = 0$ in the statement of Theorem \ref{thm:BWB}. We investigate the effect of a Weyl reflection on $\lambda$. For notational simplicity, we set $\omega_{0} = \omega_{r} = 0$. There are $r(r-1)/2$ positive roots $\beta := \beta_{i} + \beta_{i+1} + \cdots + \beta_{j} = -\omega_{i-1}+\omega_{i} + \omega_{j} - \omega_{j+1}$ with $i \le j$. When $i = j$, we have $\beta = \beta_{i} = -\omega_{i-1}+2\omega_{i} - \omega_{i+1}$. The partition $\lambda$ is singular if there is a positive root $\beta$ such that $r | (\lambda+\rho, \beta)$. If $\lambda$ is regular, the Weyl alcove containing $\lambda+\rho$ is bounded by affine hyperplanes $(\lambda+\rho, \beta) \equiv \pm 1 \; \mathrm{mod}\;r$ for positive roots, and there is a positive root $\beta$ such that $(\lambda+\rho, \beta) \equiv 1\; \mathrm{mod}\; r$. The Weyl reflection along the $\beta$ moves $\lambda+\rho$ to $\lambda +\rho - \beta$. 

\begin{lemma}\label{lem:boundedness}
Suppose that $r > 2n$. For any partitions $\lambda$ and $\mu$ with $|\lambda| = |\mu| = n$, $S_{\lambda}\cE_{p} \otimes S_{\mu}\cE_{p}^{*}$ is cohomologically bounded up to degree $n$. 
\end{lemma}

\begin{proof}
Suppose that $\mu$ has \emph{columns} (not rows!) $\mu^{1} \ge \mu^{2} \ge \cdots \ge \mu^{t} > 0$. Then the dual partition $\mu^{*}$, which gives the isomorphism $S_{\mu^{*}}\cE_{p} \cong S_{\mu}\cE_{p}^{*}$ has the columns $r - \mu^{t} \ge r - \mu^{t-1} \ge \cdots \ge r - \mu^{1}$. Since $|\mu| = n$ and $r > 2n$, the length of each columns of $\mu^{*}$ is at least $n+1$. 

We may decompose the tensor product of $\cS_{\lambda}\cE_{p}$ and $\cS_{\mu}\cE_{p}^{*}$ as a direct sum of Schur functors:
\[
	S_{\lambda}\cE_{p} \otimes S_{\mu}\cE_{p}^{*} \cong \bigoplus _{\nu}c_{\lambda, \mu^{*}}^{\nu}S_{\nu}\cE_{p}.
\] 
By the Littlewood-Richardson rule \cite[p. 121, Corollary 2]{Ful97}, $c_{\lambda, \mu^{*}}^{\nu}$ is the number of skew semistandard Young tableaux of shape $\nu / \mu^{*}$ with some extra combinatorial conditions. Because $r > 2n$, $\nu$ with a nonzero $c_{\lambda, \mu^{*}}^{\nu}$ is a partition that has the following properties. It has at most $n$ boxes of `small height' (the columns including the boxes have the height at most $n \le (r-1)/2$) and at least $tr-n$ boxes of `large height' (the columns containing the boxes have the height at least $(r-1)/2 + 1 \ge n+1$). See Figure \ref{fig:Youngdiagrams}.

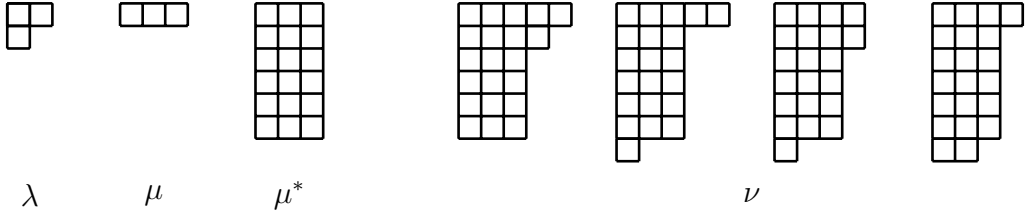
\begin{figure}[!ht]
\begin{tikzpicture}[scale=0.3]
	\draw[line width = 1pt] (-3, 0) grid (-1, 1);
	\draw[line width = 1pt] (-3, -1) grid (-2, 1);

	\draw[line width = 1pt] (2, 0) grid (5, 1);
	
	\draw[line width = 1pt] (8, 1) grid (11, -5);

	\node at (-2, -7.5) {$\lambda$};
	\node at (3.5, -7.5) {$\mu$};
	\node at (9.5, -7.5) {$\mu^{*}$};

	\draw[line width = 1pt] (17, 1) grid (20, -5);
	\draw[line width = 1pt] (20, 1) grid (22, 0);
	\draw[line width = 1pt] (20, 0) grid (21, -1);
	
	\draw[line width = 1pt] (24, 1) grid (27, -5);
	\draw[line width = 1pt] (27, 1) grid (29, 0);
	\draw[line width = 1pt] (24, -5) grid (25, -6);

	\draw[line width = 1pt] (31, 1) grid (34, -5);
	\draw[line width = 1pt] (34, 1) grid (35, -1);
	\draw[line width = 1pt] (31, -5) grid (32, -6);

	\draw[line width = 1pt] (38, 1) grid (41, -5);
	\draw[line width = 1pt] (41, 1) grid (42, 0);
	\draw[line width = 1pt] (38, -5) grid (40, -6);
	\node at (30, -7.5) {$\nu$};
\end{tikzpicture}
\caption{List of four partitions $\nu$ with nonzero $c_{\lambda, \mu^{*}}^{\nu}$ for $r = 7$, $n = 3$, $\lambda = (2 \ge 1)$, $\mu = (3)$.}\label{fig:Youngdiagrams}
\end{figure}

Let $\nu = \sum a_{j}\omega_{j}$. If $\nu$ is singular, then by Theorem \ref{thm:BWB}, $S_{\nu}\cE_{p}$ is cohomologically trivial. So, suppose that $\nu$ is regular. Then there is a positive root $\beta = \beta_{i}+ \beta_{i+1}+\cdots +\beta_{j}$ such that $(\nu+\rho, \beta) \equiv 1 \; \mathrm{mod}\; r$. Because $(\nu + \rho, \beta) = \sum_{k= i}^{j}(a_{k}+1) \le 2n + r-1 < 2r - 1$, $(\nu + \rho, \beta) = 1$ or $r+1$. The first case only occurs if $\beta = \beta_{i}$ and $a_{i} = 0$. In this case, the $i$-th coefficient of $\nu + \rho - \beta$ is $-1$, so it is out of the region of dominant weights. Thus, we do not consider the reflection. Hence, we may assume that there is a root $\beta = \beta_{i} + \beta_{i+1} + \cdots + \beta_{j}$ with $i < j$ such that $(\nu + \rho, \beta) = r + 1$. 

We claim that $i \le (r-1)/2$ and $j \ge (r-1)/2+1$. If not, $j \le (r-1)/2$ or $i \ge (r-1)/2+1$. In the former case, because $(\nu + \rho, \beta) \le \sum_{k=i}^{(r-1)/2}(a_{k} + 1) \le n + (r-1)/2$. But $n + (r-1)/2 < (2r-1)/2 < r+1 = (\nu + \rho, \beta)$, hence this is impossible. We may exclude the $i \ge (r-1)/2+1$ case similarly. 

Take the Weyl reflection associated with $\beta = \beta_{i} + \cdots + \beta_{j}$. Since $i \le (r-1)/2$ and $j \ge (r-1)/2 + 1$, in terms of the associated Young diagram, this corresponds to 
\begin{enumerate}
\item Removing a box from one row in the small height part;
\item Adding a box to one row in the large height part;
\item Eliminating any column with length $r$ (if it exists). 
\end{enumerate}
Since there are at most $n$ boxes of small height on $\nu$, this procedure terminates at most $n$ steps and reaches $0 + \rho$. Therefore, the length $\ell(\nu)$ is at most $n$. By Theorem \ref{thm:BWB}, with $h = 0$, $\rH^{i}(\cM_{X}(r, \cO), S_{\nu}\cE_{p}) = 0$ for $i > n$. 

Therefore, all irreducible factors $S_{\nu}\cE_{p}$ are cohomologically bounded up to degree $n$, and the same is true for $S_{\lambda}\cE_{p} \otimes S_{\mu}\cE_{p}^{*}$. 
\end{proof}

\begin{lemma}\label{lem:cohtrivial}
Suppose that $r > 2n$. For a partition $\lambda \vdash n$, if $|\lambda| = n$, then $\lambda$ and $\lambda^{*}$ are singular. Thus $S_{\lambda}\cE_{p}$ and $S_{\lambda^{*}}\cE_{p}$ are cohomologically trivial.
\end{lemma}

\begin{proof}
Set $\lambda = \sum a_{j}\omega_{j}$. Then $\lambda + \rho = \sum (a_{j}+1)\omega_{j}$. Since $|\lambda| = \sum ja_{j} = n$, $a := \sum a_{j} \le n$ and $a_{j} = 0$ for $j > r - a > n$. Now if we take $\beta := \beta_{1} + \cdots + \beta_{r-a}$, 
\[
	(\lambda+\rho, \beta) = \sum_{j=1}^{r-a}(a_{j}+1) = \sum_{j=1}^{r-1}(a_{j}+1) - (a-1) = a + r-1 - (a-1) = r.
\]
Therefore $\lambda$ is singular. The singularity of $\lambda^{*}$ follows from the symmetry $\beta_{j} \leftrightarrow \beta_{r-j}$. The second statement follows from Theorem \ref{thm:BWB}.
\end{proof}

\subsection{Cohomological boundedness and triviality on the moduli space}\label{ssec:cohomologyonmodulispace}

We are ready to prove the desired cohomological boundedness and the triviality over $\rM_{X}(r, L)$. In this section, we assume that $(r, d) = 1$ where $d = \det L$. 

The following lemma tells us that for `balanced' tensor products of Schur functors, the universal bundle gives the same bundle with the pull-back of a Poincar\'e bundle.

\begin{lemma}\label{lem:Poincareisuniversal}
Let $\cE$ be the universal bundle over $X \times \cM_{X}(r, L)$ and let $\sigma^{*}\cE$ be a Poincar\'e bundle over $X \times \rM_{X}(r, L)$ for a section $\sigma : \rM_{X}(r, L) \to \cM_{X}(r, L)^{s}$.  For two sequences of points $p_{1}, p_{2}, \cdots, p_{k} \in X$, $q_{1}, q_{2}, \cdots, q_{\ell} \in X$, and two sequences of partitions $\lambda^{1}, \lambda^{2}, \cdots, \lambda^{k}$ and $\mu^{1}, \mu^{2}, \cdots, \mu^{\ell}$ such that $\sum |\lambda^{i}| = \sum |\mu^{j}|$, there is an isomorphism 
\[
	\bigotimes S_{\lambda_{i}}\cE_{p_{i}} \otimes \bigotimes S_{\mu_{j}}\cE_{q_{j}}^{*} \cong 
	p^{*}\left(\bigotimes S_{\lambda_{i}}\sigma^{*}\cE_{p_{i}} \otimes \bigotimes S_{\mu_{j}}\sigma^{*}\cE_{q_{j}}^{*}\right).
\]
\end{lemma}

\begin{proof}
For the good moduli map $p : \cM_{X}(r, L)^{s} \to \rM_{X}(r, L)$, since $\cM_{X}(r, L)^{s}$ is a $\CC^{*}$-gerbe over $\rM_{X}(r, L)$, the pull-back of a universal bundle $p^{*}\sigma^{*}\cE$ differs from $\cE$ by a tensor product of a line bundle $F$ on the stack. Because the Schur functor construction is functorial, 
\[
\begin{split}
	p^{*}\left(\bigotimes S_{\lambda_{i}}\sigma^{*}\cE_{p_{i}} \otimes \bigotimes S_{\mu_{j}}\sigma^{*}\cE_{q_{j}}^{*}\right)
	&\cong \bigotimes S_{\lambda_{i}}p^{*}\sigma^{*}\cE_{p_{i}} \otimes \bigotimes S_{\mu_{j}}p^{*}\sigma^{*}\cE_{q_{j}}^{*}\\
	&\cong \bigotimes S_{\lambda_{i}}\left(\cE_{p_{i}}\otimes F\right) \otimes \bigotimes S_{\mu_{j}}\sigma^{*}\left(\cE_{q_{j}} \otimes F\right)^{*}\\
	&\cong \bigotimes \left(S_{\lambda_{i}}\cE_{p_{i}}\right) \otimes F^{|\lambda_{i}|} \otimes \bigotimes \left(S_{\mu_{j}}\cE_{q_{j}}^{*}\right) \otimes F^{-|\mu_{j}|}\\
	&\cong \bigotimes S_{\lambda_{i}}\cE_{p_{i}} \otimes \bigotimes S_{\mu_{j}}\cE_{q_{j}}^{*} \otimes F^{\sum |\lambda_{i}| - \sum |\mu_{j}|}
\end{split}
\]
and the power of $F$ is zero by the assumption.
\end{proof}

We first show the cohomological boundedness.

\begin{proposition}\label{prop:reductionboundedness}
Suppose $r > 2n$ and let $\bp \in X_{n}$. Then $\rH^{i}(\rM_{X}(r, \cO), \cF_{\bp} \otimes \cF_{\bp}^{*}) = 0$ for $i > n$. 
\end{proposition}

\begin{proof}
We first consider the case that $\bp = (np)$. Since $\cF_{\bp}$ is a deformation of $(\sigma^{*}\cE_{p})^{\otimes n}$ (Lemma \ref{lem:FMkernel}), by the semicontinuity, it is sufficient to show that $\rH^{i}(\rM_{X}(r, \cO), (\sigma^{*}\cE_{p})^{\otimes n} \otimes (\sigma^{*}\cE_{p}^{*})^{\otimes n}) = 0$ for $i > n$. By the Schur-Weyl duality, there is an isomorphism 
\[
	(\sigma^{*}\cE_{p})^{\otimes n} \cong \bigoplus_{\lambda \vdash n}V_{\lambda} \otimes S_{\lambda}(\sigma^{*}\cE_{p}),
\]
where $V_{\lambda}$ is an irreducible $S_{n}$-representation associated to $\lambda$. Then 
\[
\begin{split}
	(\sigma^{*}\cE_{p})^{\otimes n} \otimes (\sigma^{*}\cE_{p}^{*})^{\otimes n} &\cong \bigoplus_{\lambda \vdash n}V_{\lambda} \otimes S_{\lambda}(\sigma^{*}\cE_{p}) \otimes \bigoplus_{\lambda \vdash n}V_{\lambda} \otimes S_{\lambda}(\sigma^{*}\cE_{p}^{*}) \\
	&\cong \bigoplus_{\lambda, \mu \vdash n} m_{\lambda, \mu}S_{\lambda}(\sigma^{*}\cE_{p}) \otimes S_{\mu}(\sigma^{*}\cE_{p}^{*}), 
\end{split}
\]
for some integers $m_{\lambda, \mu}$. Since $p^{*} : \rD^{b}(\rM_{X}(r, L)) \to \rD^{b}(\cM_{X}(r, L)^{s})$ is fully-faithful \cite[Lemma 6.1]{LM23}, the cohomology of $S_{\lambda}(\sigma^{*}\cE_{p}) \otimes S_{\mu}(\sigma^{*}\cE_{p}^{*})$ can be computed after pulling back to $\cM_{X}(r, L)^{s}$. By Lemma \ref{lem:Poincareisuniversal}, this is isomorphic to $S_{\lambda}\cE_{p} \otimes S_{\mu}\cE_{p}^{*}$. 

Littlewood-Richardson rule implies
\[
	S_{\lambda}\cE_{p} \otimes S_{\mu}\cE_{p}^{*} \cong \bigoplus _{\nu}c_{\lambda, \mu^{*}}^{\nu}S_{\nu}\cE_{p}.
\] 
So, we may reduce the statement to the cohomological boundedness of $S_{\nu}\cE_{p}$. Then, by the cohomology identifications in previous sections, 
\begin{equation}\label{eqn:reductiontotrivialdet}
\begin{split}
	\rH^{*}(\cM_{X}(r, L)^{s}, S_{\nu}\cE_{p}) & \stackrel{\eqref{eqn:cohomologycomparison2}}{\cong} 
	\rH^{*}(\cM_{X, (p)}(r, L, \ba), L_{p, \nu}) 
	\stackrel{\eqref{eqn:pullbackoflambda}}{\cong} 
	\rH^{*}(\cM_{X, (p, p')}(r, \cO, \ba'), L_{p, \nu}) \\
	& \stackrel{\eqref{eqn:unstablelociarenegligible}}{\cong}
	\rH^{*}(\cM_{X, (p, p')}(r, \cO), L_{p, \nu})
	\stackrel{\eqref{eqn:cohomologycomparison}}{\cong} 
	\rH^{*}(\cM_{X}(r, \cO), S_{\nu}\cE_{p}).
\end{split}
\end{equation}
In other words, we may reduce the cohomology computation of $S_{\lambda}\cE_{p} \otimes S_{\mu}\cE_{p}^{*}$ over $\cM_{X}(r, \cO)$. Finally, by Lemma \ref{lem:boundedness}, over $\cM_{X}(r, \cO)$, $S_{\nu}\cE_{p}$ is cohomologically bounded up to degree $n$.

Now, we prove the general case. The symmetric product $X_{n} \cong \Hilb^{n}(X)$ is naturally stratified by the multiplicities of point configuration. For a partition $\lambda = (\lambda_{1} \ge \lambda_{2} \ge \cdots \ge \lambda_{k}>0)$ of $n$, we set 
\[
	X_{\lambda} = \{ \lambda_{1}p_{1} + \lambda_{2}p_{2} + \cdots + \lambda_{k}p_{k} \in \Hilb^{n}(X)\;|\; p_{i}\ne p_{j} \mbox{ if } i \ne j\}.
\]
Then $X_{n} = \sqcup_{\lambda \vdash n}X_{\lambda}$. By the semicontinuity of cohomology, the vanishing of cohomology is true for general points of arbitrary strata $X_{\lambda} \subset X_{n}$, as $X_{(n)}$ is the unique closed stratum and the closure of any other stratum intersect $X_{(n)}$. Theorem \ref{thm:BWB} does not depend on the actual point configuration. Therefore, if the cohomology vanishing is true for one point of $X_{\lambda}$, then it is true for every point on $X_{\lambda}$. Thus, the desired statement is true for the whole $X_{n}$.
\end{proof}

Along the same line of the proof, we can show the cohomological triviality as well.

\begin{proposition}\label{prop:triviality}
Suppose that $r > 2n$ and $\bp \ne \bq \in X_{n}$. Then $\cF_{\bp} \otimes \cF_{\bq}^{*}$ is cohomologically trivial.
\end{proposition}

\begin{proof}
First, consider $\bp = (np)$ and $\bq = (nq)$ with $p \ne q \in X$. Because of the deformation property of $\cF_{\bp}$ (Lemma \ref{lem:FMkernel}), it is sufficient to show the cohomological triviality of $(\sigma^{*}\cE_{p})^{\otimes n} \otimes (\sigma^{*}\cE_{q}^{*})^{\otimes n}$. Applying the Schur-Weyl duality and the tensor decomposition, we have 
\[
	(\sigma^{*}\cE_{p})^{\otimes n}\otimes (\sigma^{*}\cE_{q}^{*})^{\otimes n} \cong \bigoplus_{\lambda, \mu \vdash n}m_{\lambda, \mu}S_{\lambda}(\sigma^{*}\cE_{p})\otimes S_{\mu}(\sigma^{*}\cE_{q}^{*}).
\]
By virtue of Lemma \ref{lem:Poincareisuniversal}, over $\cM_{X}(r, L)^{s}$, $p^{*}(S_{\lambda}(\sigma^{*}\cE_{p})\otimes S_{\mu}(\sigma^{*}\cE_{q}^{*})) \cong S_{\lambda} \cE_{p} \otimes S_{\mu}\cE_{q}^{*}$. Then, as in \eqref{eqn:reductiontotrivialdet}, we may reduce the computation to the bundle over $\cM_{X}(r, \cO)$. Lemma \ref{lem:cohtrivial} implies that $\lambda$ is singular. By Theorem \ref{thm:BWB}, $S_{\lambda}\cE_{p} \otimes S_{\mu}\cE_{q}^{*}$ is cohomologically trivial regardless of $\mu$. Hence $(\sigma^{*}\cE_{p})^{\otimes n}\otimes (\sigma^{*}\cE_{q}^{*})^{\otimes n}$ is cohomologically trivial, too.

The semicontinuity tells us that the cohomological triviality is valid for general points on $X_{\lambda} \times X_{\mu} \subset X_{n} \times X_{n}$. Since the cohomology computation only depends on the pair of partitions by Theorem \ref{thm:BWB}, we obtained the desired triviality for all points. 
\end{proof}

In summary, we have Items (2) and (3) of Theorem \ref{thm:BondalOrlov}.

\begin{remark}\label{rmk:positiveh}
One may observe that we only used a particular case ($h = 0$) of Theorem \ref{thm:BWB}. This is because all of the bundles we need to compute their cohomology are `balanced' ones. The technique developed in this article can be extended to more general questions including the computation of a semiorthogonal decomposition in $\rD^{b}(\rM_{X}(r, L))$ and a construction of ACM bundles on $\rM_{X}(r, L)$ as in \cite{LM21, LM23}. We expect that in the future work along this direction, a twist by the theta divisor $\Theta$ will play a crucial role, and we need to employ the full power of Theorem \ref{thm:BWB}.
\end{remark}

%%%%%%%%%%%%%%%%%%%%%%%%%%%%%%%%%%%%%

\section{Simplicity}\label{sec:simple}

To complete the proof of Theorem \ref{thm:BondalOrlov}, and hence Theorem \ref{thm:embedding}, it remains to show the simplicity of $\cF_{\bp}$ for $\bp \in X_{n}$. In this section, we prove this result. In this section, we assume $(r, d) = 1$ where $d = \deg L$. Let $\sigma^{*}\cE$ be a Poincar\'e bundle over $X \times \rM_{X}(r, L)$. 

\begin{proposition}\label{prop:simple}
For any $\bp \in X_{n}$, $\cF_{\bp}$ is simple. That is, $\mathrm{End}(\cF_{\bp}) \cong \CC$. 
\end{proposition}

\begin{proof}
For any vector bundle $E$, $\mathrm{End}(E) \ne 0$. Thus, it is sufficient to show that $\dim \mathrm{End}(\cF_{\bp}) \le 1$.

Let $\bp = (\sum \lambda_{j}p_{j}) \in X_{n}$ with $p_{i} \ne p_{j}$ if $i \ne j$. Then $\cF_{\bp} \cong \bigotimes_{j}\cF_{\lambda_{j}p_{j}}$ \cite[Lemma 2.6]{TT21} and it is an $S_{\lambda_{1}} \times S_{\lambda_{2}} \times \cdots \times S_{\lambda_{k}}$-invariant bundle. Since it is a deformation of $\bigotimes_{j}(\sigma^{*}\cE_{p_{j}})^{\otimes \lambda_{j}}$, by the semicontinuity, there is an injective map 
\[
	\mathrm{Hom}(\cF_{\bp}, \cF_{\bp}) \hookrightarrow 
	\mathrm{Hom}(\cF_{\bp}, \bigotimes_{j}(\sigma^{*}\cE_{p_{j}})^{\otimes \lambda_{j}}) \cong
	\mathrm{Hom}(\cF_{\bp}, \bigotimes_{j}\bigoplus_{\mu \vdash \lambda_{j}}V_{\mu} \otimes S_{\mu}(\sigma^{*}\cE_{p_{j}})).
\]
Since the deformation is $S_{\lambda_{1}} \times S_{\lambda_{2}} \times \cdots \times S_{\lambda_{k}}$-equivariant and $S_{\lambda_{1}} \times S_{\lambda_{2}} \times \cdots \times S_{\lambda_{k}}$ trivially acts on $\cF_{\bp}$, the image factors through the invariant subbundle, that is, $\bigotimes_{j}\mathrm{Sym}^{\lambda_{j}}(\sigma^{*}\cE_{p_{j}})$. Therefore, we have an injective map
\[
	\mathrm{Hom}(\cF_{\bp}, \cF_{\bp}) \hookrightarrow 
	\mathrm{Hom}(\cF_{\bp}, \bigotimes_{j}\mathrm{Sym}^{\lambda_{j}}(\sigma^{*}\cE_{p_{j}})).
\]
We may apply the deformation argument once again, then we have an inclusion 
\[
	\mathrm{Hom}(\cF_{\bp}, \bigotimes_{j}\mathrm{Sym}^{\lambda_{j}}(\sigma^{*}\cE_{p_{j}})) \hookrightarrow
	\mathrm{Hom}(\bigotimes_{j}(\sigma^{*}\cE_{p_{j}})^{\otimes \lambda_{j}}, \bigotimes_{j}\mathrm{Sym}^{\lambda_{j}}(\sigma^{*}\cE_{p_{j}}))
\]
\[
	\cong \mathrm{Hom}(\bigotimes_{j}\bigoplus_{\mu \vdash \lambda_{j}}V_{\mu} \otimes S_{\mu}(\sigma^{*}\cE_{p_{j}}), \bigotimes_{j}\mathrm{Sym}^{\lambda_{j}}(\sigma^{*}\cE_{p_{j}}))
\]
and the last space is a direct sum of 
\[
	\rH^{0}(\rM_{X}(r, L), \bigotimes_{j}S_{\mu}(\sigma^{*}\cE_{p_{j}}^{*}) \otimes S_{(\lambda_{j})}(\sigma^{*}\cE_{p_{j}})).
\]

By applying the chain of isomorphisms in \eqref{eqn:reductiontotrivialdet}, the computation of the cohomology group is reduced to that of 
\begin{equation}\label{eqn:globalsectionforsimpleness}
	\rH^{0}(\cM_{X}(r, \cO), \bigotimes_{j}S_{\mu^{*}}(\cE_{p_{j}}) \otimes S_{(\lambda_{j})}(\cE_{p_{j}})).
\end{equation}
Note that because we are only interested in the zeroth-cohomology and the codimension of the unstable locus is large (Lemma \ref{lem:unstablecodim}), the map in \eqref{eqn:unstablelociarenegligible} is an isomorphism without the condition on the partition. 

Decompose $S_{\mu^{*}}\cE_{p_{j}} \otimes S_{(\lambda_{j})}\cE_{p_{j}}\cong \bigoplus c_{\mu^{*}, (\lambda_{j})}^{\nu}S_{\nu}\cE_{p_{j}}$. Note that among the direct summands on the right-hand side, the only one that $\nu + \rho$ is already in the positive Weyl alcove is $S_{(t, t, \cdots, t)}\cE_{p_{j}} \cong S_{0}\cE_{p_{j}} \cong \cO$, by Remark \ref{rem:positivealcovewhenhiszero}, and it appears only if $\mu = (\lambda_{j})$. By Theorem \ref{thm:BWB}, the only nonzero contribution to \eqref{eqn:globalsectionforsimpleness} is given by the direct sum of 
\[
	\rH^{0}(\cM_{X}(r, \cO), \bigotimes_{j}S_{(\lambda_{j})^{*}}(\cE_{p_{j}})\otimes S_{(\lambda_{j})}(\cE_{p_{j}})), 
\]
and its multiplicity is precisely the product of the multiplicities of $S_{(\lambda_{j})}(\cE_{p_{j}}) \cong \mathrm{Sym}^{\lambda_{j}}\cE_{p_{j}}$ in $(\cE_{p_{j}})^{\otimes \lambda_{j}}$, which is one. Therefore, 
\[
	\dim\mathrm{End}(\cF_{\bp}) \le \dim \rH^{0}(\cM_{X}(r, \cO), \bigotimes_{j}S_{\mu^{*}}(\cE_{p_{j}}) \otimes S_{(\lambda_{j})}(\cE_{p_{j}})) = 1
\]
and we are done.
\end{proof}

This completes the proof of Theorem \ref{thm:BondalOrlov} and Theorem \ref{thm:embedding}.

\section{Fano visitor}\label{sec:Fanovisitor}

\subsection{Fano visitors for derived categories}

In this section, we will show that $X_n$ and $\Jac(X)$ are Fano visitors. Since $\rM_{X}(r, L)$ is a smooth Fano variety of dimension $(r^{2}-1)(g-1)$, Theorem \ref{thm:embedding} immediately implies the following consequence.

\begin{corollary}\label{cor:symmetricproductFanovisitor}
Let $X$ be a smooth curve of genus $g \ge 2$. Then the symmetric product $X_{n}$ is a Fano visitor, and its Fano dimension $\mathrm{Fdim} \;X_{n}$ is at most $((2n+1)^{2}-1)(g-1)$. 
\end{corollary}

\begin{remark}
The upper bound of the Fano dimension is not tight. For $n \le g-1$, the embedding $\rD^{b}(X_{n}) \to \rD^{b}(\rM_{X}(2, L))$ is obtained by the work of Tevelev-Torres in \cite{TT21}. Thus, for this range, the upper bound of the Fano dimension is $\dim \rM_{X}(2, L) = 3(g-1)$, thus $\Fdim X_{n} \le 3g - 3$. For general low genus curves, one can find better upper bounds \cite[Section 5]{KL23}. 
\end{remark}

For any birational morphism $f : V \to W$ between two smooth projective varieties, the natural morphism $\cO_{W} \to R f_{*}\cO_{V}$ is an isomorphism \cite[p. 144]{Hir64}. Thus, for any $F^{\bullet}, G^{\bullet} \in \rD^{b}(W)$, 
\[
	\mathrm{Hom}(L f^{*}F^{\bullet}, L f^{*}G^{\bullet}) \cong \mathrm{Hom}(F^{\bullet}, Rf_{*}Lf^{*}G^{\bullet}) \cong \mathrm{Hom}(F^{\bullet}, G^{\bullet})
\]
by the projection formula. Thus, $L f^{*} : \rD^{b}(W) \to \rD^{b}(V)$ is fully-faithful. 

\begin{proof}[Proof of Theorem \ref{thm:Fanovisitor}]
The \emph{Abel-Jacobi map} $AJ : X_{g} \to \mathrm{Jac}(X)$ is a birational morphism between smooth varieties. By the above argument, we have an embedding $L AJ^{*} : \rD^{b}(\mathrm{Jac}(X)) \to \rD^{b}(X_{g})$. So we have an embedding 
\[
	\Phi_{\cF} \circ L AJ^{*} : \rD^{b}(\mathrm{Jac}(X)) \to \rD^{b}(\rM_{X}(2g+1, L)). 
\]
The dimension of $\rM_{X}(2g+1, L)$ is $((2g+1)^{2}-1)(g-1) = 4(g+1)g(g-1)$. So we obtain an upper bound of $\mathrm{Fdim} \; \mathrm{Jac}(X)$.
\end{proof}

\begin{remark}
If $g(X) = 0$, the Jacobian is trivial. When $g(X) = 1$, the Abel-Jacobi map is an isomorphism. 
\end{remark}

It is well-known that any principally polarized abelian variety $A$ of dimension $\le 3$ are either Jacobian of a curve or products of Jacobians. Thus, applying \cite[Corollary 5.10]{Kuz11}, $\rD^{b}(A)$ is embedded into the derived category of a product of Fano varieties, which is Fano. 

\begin{corollary}\label{cor:ppavFanovisitor}
Any principally polarized abelian varieties of dimension $\le 3$ are Fano visitors. 
\end{corollary}

\begin{remark}\label{rmk:newFanovisitor}
Since $X_{n}$ and abelian varieties are not simply connected, they are not complete intersections. Thus, they are new examples of Fano visitors. 
\end{remark}

\begin{remark}\label{rmk:indecomposability}
Derived categories of abelian varieties are also indecomposable \cite[Corollary 1.5]{KO15}. In the range of $1 \le n \le g-1$, the indecomposability of $\rD^{b}(X_{n})$ is obtained by an accumulation of many works (see \cite{Lin21}). On the other hand, when $n \ge g$, an explicit semiorthogonal decomposition of $\rD^{b}(X_{n})$ is obtained in \cite[Section 5.5]{Tod21}. Indeed, $\rD^{b}(X_{n})$ can be decomposed into the derived categories of $\mathrm{Jac}(X)$ and $X_{m}$ with $m \le g-1$. 
\end{remark}

\subsection{Motivic Fano visitors}

Orlov conjectured that semiorthogonal decompositions of derived categories of algebraic varieties are closely related to motives of them \cite{Orl05}. From this perspective, it is natural to ask the following definition.

\begin{definition}
A smooth projective variety $V$ is a \emph{motivic Fano visitor} if its rational Chow motive (or one of its tensor products with Lefschetz motives) is a direct summand of the rational Chow motive of a smooth Fano variety $W$.
\end{definition}

Then we can state the Fano visitor problem for motives as follows.

\begin{question}[Fano visitor problem for motives]
Is every smooth projective variety a motivic Fano visitor?
\end{question}

One may ask similar questions for the other versions of motives, for instance, the Grothendieck ring of varieties, numerical motives, Voevodsky motives, and so on. For simplicity, we restrict ourselves to rational Chow motives in this paper. See \cite{GL20} and references therein for more details about motives.

Del Ba\~no obtained the following formula.

\begin{theorem}[\protect{\cite{Ban98}\cite[Theorem 4.11]{Ban01}\label{Bano;inversion}}]
Let $r \geq 2, d$ be two integers which are coprime to each other. Then the motivic Poincar\'e polynomial of $\rM_{X}(r,d)$ is
$$ \sum_{s=1}^r \sum_{r_1+\cdots+r_s=r, r_i \in \NN} (-1)^{s-1} \frac{((1+1)^{h^1(C)})^s}{(1-\LL)^{s-1}} \prod_{j=1}^s \prod_{i=1}^{r_j-1} \frac{(1+\LL^{ i})^{h^1(C)}}{(1-\LL^{ i})(1-\LL^{i+1})} $$
$$  \prod_{j=1}^{s-1}\frac{1}{1-\LL^{r_j+r_{j+1}}}  \LL^{ (\sum_{i < j} r_i r_j (g-1)+\sum_{i=1}^{s-1}(r_i+r_{i+1})\langle -\frac{r_1+\cdots+r_i}{r}d \rangle)}. $$
\end{theorem}

Taking the quotient by the motive of the Jacobian, we obtain an analogous formula for $\rM_{X}(r, L)$. See \cite{GL20} for the detail of the argument. 

Fu, Hoskins, and Lehalleur proved that under some finiteness assumption, the comparison of the motivic Poincar\'e polynomial suffices to obtain an isomorphism between Chow motives. 

\begin{proposition}[\protect{\cite[Proposition 2.2]{FHL21}}]\label{prop:equivalenceofeffectivemotives}
Let $M_1, M_2$ be two effective Chow motives which are Kimura finite dimensional whose motivic Poincar\'e polynomials are the same. Then $M_1$ is isomorphic to $M_2$ as effective Chow motives.
\end{proposition}

Using the strategies of \cite{GL20} and \cite{FHL21}, we get the following proposition.

\begin{proposition}\label{prop:XnismotivicFanovisitor}
For any smooth curve $X$ and $n \in \mathbb{N}$, the symmetric product $X_{n}$ is a motivic Fano visitor.
\end{proposition}

\begin{proof}
Let $h(\rM_X(r,L))$ be the Chow motive of $\rM_X(r,L).$ It admits the following decomposition
\[
	h(\rM_X(r,L)) = h_{-}(\rM_X(r,L)) \oplus h_{0}(\rM_X(r,L)) \oplus h_{+}(\rM_X(r,L)) 
\]
which provides the following cohomology decomposition,
\[
	\rH^*(\rM_X(r,L)) = \rH^{i \leq 4n}(\rM_X(r,L)) \oplus \rH^{4n < i < (r^2-1)(g-1)-4n}(\rM_X(r,L)) \oplus \rH^{i \geq (r^2-1)(g-1)-4n}(\rM_X(r,L))
\]
after applying the realization functor \cite{Ban98}.

When $r$ is sufficiently large, Theorem \ref{Bano;inversion}, the proof of \cite[Theorem 1.3]{GL20}, and Proposition \ref{prop:equivalenceofeffectivemotives} imply that $h_{-}(\rM_X(r,L))$ admits a decomposition $h_{-}(\rM_X(r,L)) = h(X_n) \otimes \LL^{\otimes n} \oplus R$ where $R$ is an effective Chow motive. Using Proposition \ref{prop:equivalenceofeffectivemotives} again, we obtain the desired conclusion.
\end{proof}

Vial proved that if there is a surjective morphism $f : W \to V$ between two smooth projective varieties, the rational Chow motive of $V$ is a direct summand of that of $W$ \cite[Theorem 1.4]{Via13}. Note that the \emph{Albanese map} $Alb : X_{n} \to \mathrm{Jac}(X)$ is surjective for any smooth curve $X$ when $n \geq g.$ Thus, we obtain the following:

\begin{corollary}\label{cor:jacobianmotivicfanovisitor}
For any smooth curve $X$, the Jacobian $\Jac(X)$ is a motivic Fano visitor.
\end{corollary}

On the other hand, Matsusaka's theorem says that every abelian variety admits a surjection from $\mathrm{Jac}(X)$ for some curve $X$ \cite{Mat52}. From Matsusaka's theorem and Vial's theorem, we have the following conclusion, which provides a shred of evidence for the affirmative answer to Question \ref{que:abelian}. 

\begin{corollary}\label{cor:abelianvarietymotivicFanovisitor}
Every abelian variety is a motivic Fano visitor.
\end{corollary}

%%%%%%%%%%%%%%%%%%%%%%%%%%%%%%%%%%%%%%%%

\bibliographystyle{alpha}

\begin{thebibliography}{KKLL17}

\bibitem[Ban98]{Ban98} S. del Ba\~no. \emph{On motives and moduli spaces of stable vector bundles over a curve.} Thesis, 1998.

\bibitem[Ban01]{Ban01} S. del Ba\~no. \emph{On the Chow motive of some moduli spaces.} J. Reine Angew. Math. 532 (2001), 105--132.

\bibitem[BL94]{BL94}
{A. Beauville and Y. Laszlo}, Conformal blocks and generalized theta functions. {\em Commun. Math. Phys.}, 164, 385--419, 1994.

\bibitem[BGM23]{BGM23}
{P. Belmans, S. Galkin, and S. Mukhopadhyay}, Decompositions of moduli spaces of vector bundles and graph potentials.
{\em Forum Math. Sigma}, 11 (2023), Paper No. e16, 28 pp.

\bibitem[BM19]{BM19}
{P. Belmans and S. Mukhopadhyay}, Admissible subcategories in derived categories of moduli of vector bundles on curves. {\em Adv. Math.} 351 (2019), 653--675. 

\bibitem[BBF16]{BBF16}
{M. Bernardara, M. Bolognesi, and D. Faenzi.} Homological projective duality for determinantal varieties. {\em Adv. Math.} 296 (2016), 181--209.

\bibitem[Bho89]{Bho89}
{U. Bhosle}, Parabolic vector bundles on curves. 
{\em Ark. Mat.} 27 (1989), no. 1, 15--22. 

\bibitem[BO95]{BO95}
{A. Bondal and D. Orlov}, Semiorthogonal decomposition for algebraic varieties. Preprint, arXiv:alg-geom/9506012, 1995.

\bibitem[CH11]{CH11}
{M. Casanellas and R. Hartshorne}, ACM bundles on cubic surfaces. {\em J. Eur. Math. Soc. (JEMS)} 13 (2011), no. 3, 709--731. 

\bibitem[FHL21]{FHL21}
{L. Fu, V. Hoskins, S. P. Lehalleur}, Motives of moduli spaces of rank 3 vector bundles and Higgs bundles on a curve. {\em Electronic Research Archive}
2022, Volume 30, Issue 1: 66--89.

\bibitem[FK18]{FK18}
{A. Fonarev and A. Kuznetsov}, Derived categories of curves as components of Fano manifolds. {\em J. London Math. Soc.} (2) 97 (2018) 24--46.
%

\bibitem[Ful97]{Ful97}
{W. Fulton}, {\em Young tableaux. With applications to representation theory and geometry}. London Mathematical Society Student Texts, 35. Cambridge University Press, Cambridge, 1997. x+260 pp. ISBN: 0-521-56144-2.

\bibitem[GL20]{GL20}
{T. G\'omez and K.-S. Lee}, Motivic decompositions of moduli spaces of vector bundles on curves. Preprint, arXiv:2007.06067.  

\bibitem[HL15]{HL15}
{D. Halpern-Leistner}, The derived category of a GIT quotient. {\em J. Amer. Math. Soc.} 28 (2015), no. 3, 871--912.

\bibitem[Hir64]{Hir64}
{H. Hironaka},
Resolution of singularities of an algebraic variety over a field of characteristic zero. I. Ann. of Math. (2) 79 (1964), 109--203. 

\bibitem[KO15]{KO15}
{K. Kawatani, S. Okawa}, Nonexistence of semiorthogonal decompositions and sections of the canonical bundle. Preprint, arXiv:1508.00682.

\bibitem[KKLL17]{KKLL17}
{Y.-H. Kiem, I.-K. Kim, H. Lee, and K.-S. Lee}, 
All complete intersection varieties are Fano visitors. 
Adv. Math. 311 (2017), 649--661. 

\bibitem[KL23]{KL23}
{Y.-H. Kiem and K.-S. Lee},
Fano visitors, Fano dimension and Fano orbifolds. {\em Proceedings for the Moscow-Shanghai-Pohang conferences.} Springer Proc. Math. Stat., 409
Springer, Cham, 2023, 517--544.

\bibitem[Kuz11]{Kuz11}
{A. Kuznetsov}, Base change for semiorthogonal decompositions.
{\em Compos. Math.} 147 (2011), no. 3, 852--876. 

\bibitem[Kuz19]{Kuz19}
{A. Kuznetsov}, 
Embedding derived categories of an Enriques surface into derived categories of Fano varieties. 
{\em Izv. Ross. Akad. Nauk Ser. Mat.} 83 (2019), no. 3, 127--132.

\bibitem[Lee18]{Lee18}
{K.-S. Lee}, Remarks on motives of moduli spaces of rank 2 vector bundles on curves. Preprint, arXiv:1806.11101.

\bibitem[LM21]{LM21}
{K.-S. Lee and H.-B. Moon}, Positivity of the Poincar\'e bundle on the moduli space of vector bundles and its applications. Preprint, arXiv:2106.04857.

\bibitem[LM23]{LM23}
{K.-S. Lee and H.-B. Moon}, Derived category and ACM bundles of moduli space of vector bundles on a curve. Forum. Math, Sigma. Volume 11, 2023, e81
DOI: \url{https://doi.org/10.1017/fms.2023.75}.

\bibitem[LN21]{LN21}
{K.-S. Lee and M. S. Narasimhan}, Symmetric products and moduli spaces of vector bundles of curves. Preprint, arXiv:2106.04872.

\bibitem[Lin21]{Lin21}
{X. Lin}, On nonexistence of semi-orthogonal decompositions in algebraic geometry. Preprint, arXiv:2107.09564.

\bibitem[Mat52]{Mat52}
{T. Matsusaka}, On a generating curve of an Abelian variety.
{\em Natur. Sci. Rep. Ochanomizu Univ.} 3(1952), 1--4.

\bibitem[MS80]{MS80}
{V. B. Mehta and C. S. Seshadri.} Moduli of vector bundles on curves with parabolic structures. {\em Math. Ann.} 248 (1980), no. 3, 205--239.

\bibitem[MY20]{MY20}
{H.-B. Moon and S.-B. Yoo}. Finite generation of the algebra of type A conformal blocks via birational geometry II: higher genus. {\em Proc. Lond. Math. Soc.}, vol.120 (2020), issue 2, 242--264.

\bibitem[MY21]{MY21}
{H.-B. Moon and S-B. Yoo}, Finite generation of the algebra of type A conformal blocks via birational geometry. {\em Int. Math. Res. Not. IMRN}, (2021), no. 7, 4941--4974.

\bibitem[Nar17]{Nar17}
{M. S. Narasimhan}, Derived categories of moduli spaces of vector bundles on curves. {\em J. Geom. Phys.} 122 (2017), 53--58. 

\bibitem[Nar18]{Nar18}
{M. S. Narasimhan}, Derived categories of moduli spaces of vector bundles on curves II. {\em Geometry, algebra, number theory, and their information technology applications}, 375-382, Springer Proc. Math. Stat., 251, Springer, Cham, 2018.

\bibitem[Ola21]{Ola21}
{N. Olander.} Fully faithful functors and dimension, preprint, arXiv:2110.09256.

\bibitem[Orl05]{Orl05}
{D. Orlov}, Derived categories of coherent sheaves, and motives. {\em Uspekhi Mat. Nauk} 60 (2005), no. 6(366), 231--232.

\bibitem[Orl09]{Orl09}
{D. Orlov}, Remarks on generators and dimensions of triangulated categories. {\em Mosc. Math. J.} 9 (2009), no. 1, 153--159

\bibitem[Pau96]{Pau96} 
{C. Pauly}, Espaces de modules de fibr{\'e}s paraboliques et blocs conformes. {\em Duke Math. J.}, 84(1):217--235, 1996.

\bibitem[Ram73]{Ram73}
{S. Ramanan}, The moduli spaces of vector bundles over an algebraic curve. {\em Math. Ann.} 200 (1973), 69--84.

\bibitem[Tel98]{Tel98}
{C. Teleman}, Borel-Weil-Bott theory on the moduli stack of $G$-bundles over a curve. {\em Invent. Math.} 134 (1998), no. 1, 1--57.

\bibitem[Tev23]{Tev23}
{J. Tevelev}, Braid and phantom. Preprint, arXiv:2304.01825.

\bibitem[TT21]{TT21}
{J. Tevelev and S. Torres}, The BGMN conjecture via stable pairs. Preprint, arXiv:2108.11951.

\bibitem[Tha96]{Tha96}
{M. Thaddeus}, Geometric invariant theory and flips. {\em J. Amer. Math. Soc.} 9 (1996), no. 3, 691--723. 

\bibitem[Tod21]{Tod21}
{Y. Toda}, Semiorthogonal decompositions of stable pair moduli spaces via $d$-critical flips. {\em J. Eur. Math. Soc. (JEMS)} 23 (2021), no. 5, 1675--1725.

\bibitem[Sun00]{Sun00}
{X. Sun}, Degeneration of moduli spaces and generalized theta functions, {\em J. Algebraic Geom.} 9 (2000) 459--527.

\bibitem[Via13]{Via13}
{C. Vial}, Algebraic cycles and fibrations. {\em Doc. Math}. 18 (2013), 1521--1553. 

\end{thebibliography}

\end{document}